\pdfoutput=1

\documentclass[reqno,a4paper]{amsart}
\usepackage{microtype}
\usepackage{color}
\usepackage{amssymb}
\usepackage{stmaryrd}
\usepackage[utf8]{inputenc}
\usepackage{a4wide}
\usepackage{paralist}
\usepackage{enumitem}
\usepackage{ifthen}
\usepackage{graphicx}
\usepackage{slashed}
\usepackage{comment}
\usepackage{multicol}

\usepackage{amsmath}
\usepackage{amsthm}
\usepackage{amsfonts}
\usepackage{sseq}
\usepackage{array}
\usepackage[all]{xy}
\usepackage{color}
\definecolor{darkblue}{rgb}{0,0,0.9}
\definecolor{lightblue}{rgb}{.5,.5,.9}
\usepackage[pdftex,breaklinks,colorlinks,filecolor=blue,linkcolor=blue,citecolor=blue,urlcolor=blue,hypertexnames=true,plainpages=false]{hyperref}

\usepackage{tikz}
\usepackage{etoolbox}
\usepackage{comment}

\numberwithin{equation}{section}

\newcommand\C{\mathbb{C}}

\newcommand\R{\mathbb{R}}

\newcommand\Q{\mathbb{Q}}
\newcommand\Z{\mathbb{Z}}
\newcommand\Ha{\mathbb{H}}

\newcommand{\boc}{\beta^{\Z/\!2}}
\newcommand{\bocj}[1]{\beta^{\Z/\!2^{#1}}}
\newcommand{\bocqz}{\beta^{\Q/\Z}}

\newcommand{\pt}{\textup{pt}}
\newcommand{\ind}{\textup{ind}}

\newcommand{\an}[1]{\langle{#1}\rangle}
\newcommand{\wt}{\widetilde}
\newcommand{\wh}{\widehat}
\newcommand{\ol}{\overline}

\newcommand{\xra}{\xrightarrow}

\newcommand{\mdot}{M^\circ}
\newcommand{\Sq}{Sq}
\newcommand{\MSpin}{\mathbf{MSpin}}
\newcommand{\K}{\mathbf{K}}
\newcommand{\M}{\mathbf{M}}
\renewcommand{\P}{\mathbf{P}}

\theoremstyle{plain}
\newtheorem{theorem}{Theorem}[section]
\newtheorem{lemma}[theorem]{Lemma}
\newtheorem{corollary}[theorem]{Corollary}
\newtheorem{proposition}[theorem]{Proposition}

\theoremstyle{definition}

\theoremstyle{remark}
\newtheorem{remark}[theorem]{Remark}

\usepackage{color}

\newcommand{\remdc}[1]{\begingroup\color{blue}#1\endgroup}

\advance \topmargin by -\headheight
\advance \topmargin by -\headsep
     
\evensidemargin \oddsidemargin
\title{The existence of contact structures on $9$-manifolds}
\author{Diarmuid Crowley and Huijun Yang}
\date{\today}

\begin{document}
\maketitle

\begin{abstract}
We give necessary and sufficient conditions for 
a closed orientable $9$-manifold $M$ to admit an almost contact structure.
The conditions are stated in terms of the Stiefel-Whitney classes of $M$ and
other more subtle homotopy invariants of $M$.

By a fundamental result of Borman, Eliashberg and Murphy,
$M$ admits an almost contact structure if and only if $M$ admits an over-twisted contact structure.
Hence we give necessary and sufficient conditions for $M$ to admit 
an over-twisted contact structure and we prove that
if $N$ is another closed $9$-manifold which is homotopy equivalent to $M$,
then $M$ admits an over-twisted contact structure if and only if $N$ does.

In addition, for $W_{i}(M)$ the $i$th integral Stiefel-Whitney class of $M$,
we prove that if $W_3(M) = 0$ then $W_7(M) = 0$.
\end{abstract}


\section{Introduction} 
\label{s:intro}
\enlargethispage{1\baselineskip}
Let $n = 2q{+}1$ be odd and $M$ a closed smooth orientable $n$-manifold.
If the structure group of $TM$, the tangent bundle of $M$, can be reduced
to $U(q) \subset O(2q{+}1)$, then $M$ admits an almost contact structure.
The problem of finding necessary and sufficient conditions for an orientable $n$-manifold
to admit an almost contact structure is motivated by contact geometry.
In dimensions $n \geq 5$, this problem has been given extra significance
by the ground-breaking result of Borman, Eliashberg and Murphy \cite[Theorem 1.1]{bem15}, 
which states that every almost contact structure on $M$ is homotopic to an over-twisted contact
structure.

In \cite[Lemma 2.17(1)]{bcs14} Bowden, the first author and Stipsicz showed that $M$
admits an almost contact structure if and only if $M$ admits a stable complex structure.
Hence, in terms of $K$-theory, the problem is to
determine necessary and sufficient conditions for $\wt{TM} \in \wt{KO}(M)$, the reduced real $K$-theory 
class of the stable tangent bundle of $M$, to lie in the image the real reduction homomorphism,
\[
r \colon \wt{K}(M) \to \wt{KO}(M),
\]
from reduced complex $K$-theory. 
Ideally these conditions should be stated in terms of  
computable invariants of $M$, like characteristic classes of $\wt{TM}$.  
This problem has been solved in dimensions $n \leq 7$ 
(see for instance Gray \cite[p. 433, Corollary]{gr59} and Massey \cite[p. 564]{ma61})
and 
for $3$-connected $9$-manifolds \cite{ge97} Geiges obtained 
a partial solution. 
In this paper we present a solution 
for all $9$-manifolds.

\subsection{Statement of results:}
We establish necessary notation. 
For a $CW$-complex $X$,
$H^*(X)$ denotes the integral cohomology ring of $X$,
$\rho_{2^j} \colon H^*(X) \to H^*(X; \Z/2^j)$ denotes reduction mod~$2^j$ for 
a positive integer $j$ and
%
%
%
%
\begin{equation}\label{eq:bseq}
\cdots \rightarrow H^{i}(X) \xrightarrow{~\times 2^j~} H^{i}(X)\xrightarrow{~\rho_{2^j}~} 
H^{i}(X;\Z/2^j)\xrightarrow{~\bocj{j}~} H^{i+1}(X) \to \cdots
\end{equation}
is the Bockstein sequence associated to the coefficient sequence
$\Z\xrightarrow{\times 2^j}\Z\rightarrow\Z/2^j$.
The $k$th Steenrod square is 
$\Sq^{k} \colon H^{i}(X; \Z/2) \rightarrow H^{i+k}(X; \Z/2)$
%
%
%
%
and as in Thomas \cite[p.\,\,888]{th67}, $\Omega$ denotes the normalized (un-stable) secondary cohomology operation associated to the relation
\begin{equation}\label{eq:squ}
\Sq^{2} ( \Sq^{3}  ( \rho_{2} (u) ) ) =0,
\end{equation}
for all $u \in H^4(X)$.
More precisely 
$\Omega$ is defined on the subset
\[\mathcal{D}_\Omega := \mathrm{Ker} \left[ \boc \circ \Sq^{2} \circ \rho_{2}\colon H^{4}(X)\rightarrow H^{7}(X) \right]\]
and for any $u \in \mathcal{D}_\Omega$, $\Omega(u)$ is a coset of 
$\Sq^{2} ( \rho_{2} (H^{6}(X)) ) \subseteq H^{8}(X;\Z/2)$.  Hence the operation $\Omega$ is a map
\[\Omega\colon \mathcal{D}_\Omega \rightarrow H^{8}(X;\Z/2)/\Sq^{2} ( \rho_{2} ( H^{6}(X) ) )
\]
and we denote the coset of $\Sq^{2} ( \rho_{2} (H^{6}(X)) )$ represented by $z \in H^{8}(X; \Z/2)$ as $[z]$.

As usual, $w_i(M) \in H^i(X; \Z/2)$
denotes the $i$th Stiefel-Whitney class of $M$ and $W_{i+1}(M) := \boc(w_i(M)) \in H^{i+1}(M)$
denotes the $(i{+}1)$st integral Stiefel-Whitney class of $M$.
Hence $M$ is spinnable
if and only if $w_2(M) = 0$ and is spin$^c$-able if and only if $W_3(M) = 0$.
In this case we let $c \in H^2(M)$ denote an integral lift of $w_2(M)$.
From a choice of $c$, there is a well-defined class
$$ p_c(M) \in H^4(M)$$
%
%
%
with $2 p_c(M) = c^2 - p_1(M)$,
which is defined in Equation \eqref{eq:pc}.
In Lemma~\ref{lem:omegac} we show that
$$\Omega(p_c(M)) \in H^8(M; \Z/2)/Sq^2(\rho_2(H^6(M)))$$
is independent of the choice of $c$ (Lemma~\ref{lem:omegac})
and in Proposition \ref{prop:htpopc} we show that 
$\Omega(p_c(M))$ is a homotopy invariant of $M$.
If in addition $M$ is spinnable $9$-manifold, then a further homotopy invariant,
$$ \wh \phi(M) \in H^5(M; \Z/2),$$
which an unstable tangential invariant of $M$, was defined by the first author in \cite{c20}, building on invariants of Wall \cite{w67} \and Stolz \cite{stol85}; see Section \ref{ss:phi}.

The following result gives necessary and sufficient conditions for $M$ to admit a stable
complex structure, equivalently a contact structure, in terms of the invariants mentioned above.  
It is the main result of this paper.

\begin{theorem}\label{thm:main}
Let $M$ be a closed orientable $9$-manifold. 
Then $M$ admits a contact structure if and only if one of the following 
mutually exclusive conditions holds:
\begin{enumerate}
\item[(a)] 
$w_2(M) =  0$, $w_8(M) = 0$ and $w_4(M) \cup \wh \phi(M) = 0$, or
%
\item[(b)] $w_{2}(M)\neq0$, $W_3(M) = 0$ and
$\Omega\!\left( p_c(M) \right) = 0$
for any integral lift $c$ of $w_2(M)$.
\end{enumerate}
\end{theorem}





We now discuss some applications and elaborations of Theorem \ref{thm:main}.
As background to the first application, recall that Sutherland \cite{su73} has given
examples of homotopy equivalent $17$-manifolds $M$ and $N$, where $M$ admits a 
stable complex structure but $N$ does not. 
In contrast, since all the invariants involved in the statement 
of Theorem \ref{thm:main} are homotopy invariants, for $9$-manifolds we have

\begin{corollary} \label{cor:hi}
Let $M$ and $N$ be homotopy equivalent closed orientable $9$-manifolds.
Then $M$ admits a contact structure if and only if $N$ does. 
\hfill \qed
\end{corollary}
%


%
%

The utility of a theorem like Theorem \ref{thm:main} resides in the accessibility
and properties of the invariants used in its statement.
Regarding $\wh \phi(M)$, following \cite{c20}, we set
$$\sigma_{w_4}(M) : = w_4(M) \cup \wh \phi(M) \in \Z/2$$
and in Proposition \ref{prop:pi_and_sigma} we show that 
%
%
%
\begin{equation} \label{eq:sigma_and_pi}
 \sigma_{w_4}(M) 
 = \pi^1(M, s),
\end{equation} 
where $s$ is any spin structure on $M$ and $\pi^1(M, s)$ is the index invariant 
of Anderson, Brown and Peterson \cite[\S 4]{abp66},
which is defined by evaluating $\wt{TM} \in \wt{KO}(M)$ on any $KO$-orientation 
of $(M, s)$; see Section \ref{ss:K-theory}.  
Theorem \ref{thm:main} and Equation \eqref{eq:sigma_and_pi} then show that
the existence of a contact structure on a spin $9$-manifold $(M, s)$ with $w_8(M) = 0$
is determined by the index of a twisted Dirac operator; see Lemma \ref{lem:index}.

%
%
It is not hard to check that the 
invariants of Theorem \ref{thm:main} are additive under
connected sum; see Lemma \ref{lem:Omegap}(a) for the additivity of $\Omega(p_c(M))$.
Hence if $M \sharp N$ denotes connected sum of oriented $9$-manifolds $M$ and $N$, 
we have

\begin{corollary}
\label{coro:sum}
Let $M$ and $N$ be closed spin$^c$-able $9$-manifolds.
\begin{enumerate}
\item[(a)] 
If $w_2(M) = 0$ and $w_2(N) = 0$,
then $M \sharp N$ admits a contact structure if and only if $w_8(M) = w_8(N) = 0$ and 
$\sigma_{w_4}(M) = \sigma_{w_4}(N)$.

\item[(b)] if both $M$ and $N$ are not spinnable, then $M \sharp N$ admits a contact structure if and only if both $M$ and $N$ admit contact structures;

\item[(c)] if $M$ is spinnable while $N$ is not, then $M \sharp N$ admits a contact structure if and only if $w_8(M) = 0$ and $N$ admits a contact structure.
\hfill
\qed
\end{enumerate}
\end{corollary}

%

If $w_4(M) = 0$, then the conditions of Theorem \ref{thm:main} simplify to give

\begin{corollary}\label{coro:w40}
Let $M$ be a closed orientable $9$-manifold with $w_{4}(M) = 0$.  Then $M$ admits a 
contact structure if and only if $W_3(M) = 0$.
\end{corollary}

\begin{proof}
If 
$w_2(M) = 0$, then $w_8(M) = w_4(M)^2 = 0$ by Equation \eqref{eq:w8} and
the corollary follows from Theorem \ref{thm:main}.
If 
$w_2(M) \neq 0$, then as we prove in Section~\ref{ss:o8}, $\rho_2(p_c(M)) = w_4(M) = 0$.
Hence $p_c(M) = 2x$ for some $x \in H^{4}(M)$. 
Now Thomas has shown that $\Omega(2x) = [\rho_{2}(x^{2})]$ by Lemma \ref{lem:Omegap}
and by Lemma \ref{lem:sq}(e) we have $\rho_{2}(x^{2}) \in \Sq^{2}(\rho_{2} ( H^{6}(M)) )$.
Hence $\Omega(p_c(M)) = 0$ and the corollary follows from 
Theorem \ref{thm:main}.
\end{proof}






In the non-spinnable case, 
evaluating $\Omega(p_c(M))$ seems to be a subtle problem
in general but we have obtained Theorem \ref{thm:spinc} below, which is a partial result in this direction.
Let $TH^{3}(M)$ the torsion subgroup of $H^{3}(M)$ and set
\begin{equation*}
\mathcal{D}_{M} := \{x\in H^{1}(M; \Z/2) ~|~ x \cdot w_{2}(M) \in \rho_{2}(TH^{3}(M))\}.
\end{equation*}
If $M$ is spin$^c$-able, then by Theorem \ref{thm:W7} below, $w_6(M)$ has an integral 
lift $v$ and by Proposition \ref{prop:sw} $\rho_2(cv) = 0$. 
Let $cv/2$ ambiguously
denote any class in $H^8(M)$ with $2 (cv/2) = 2$.  Just prior to Theorem \ref{thm:betad}
we show that if $\boc (x) =0$ for all $x \in \mathcal{D}_{M}$ then
 $[w_8 -\rho_2(cv/2)]$ does not depend on the choices of $c, v$
or $cv/2$.

\begin{theorem}\label{thm:spinc}
Let $M$ be a $9$-dimensional spin$^{c}$-able manifold with $w_{2}(M) \neq 0$. 
Suppose that $H_1(M)$ has no $2$-torsion, 
or more generally that $\boc(x) = 0$ for all $x \in \mathcal{D}_{M}$.
Then $M$ admits a contact structure if and only if 
$$ [w_{8}(M) - \rho_{2}(cv/2)] = [0] \in H^{8}(M; \Z/2) / \Sq^{2} ( \rho_{2} ( H^{6}(M) ) ),$$
for $c \in H^{2}(M)$ and $v \in H^{6}(M)$ integral lifts of
$w_2(M)$ and $w_6(M)$ respectively.
\end{theorem}

As we explain in Section \ref{ss:outline}, the proof of Theorem \ref{thm:main}
requires the following result about the integral Stiefel-Whitney class $W_7(M)$, 
which is of independent interest.

\begin{theorem}\label{thm:W7}
If $M$ is a closed orientable $9$-dimensional spin$^{c}$-able manifold, then
$W_{7}(M)=0$.
\end{theorem}

The following remarks provide a wider context for Theorem \ref{thm:W7}.

\begin{remark}
Let $N$ be a closed orientable smooth manifold. It follows from the definition of $W_{7}(N)$ that 
$W_{7}(N)$ is a torsion element and 
\begin{equation*} \label{eq:W7w7}
\rho_{2} (W_{7}(N)) = w_{7}(N).
\end{equation*}
Hence $W_{7}(N) = 0$ if $\dim N \le 7$.
Moreover, Massey \cite[Theorem 2]{ma62} proved that $W_{7}(N) = 0$ when $\dim N = 8$. 
\end{remark}
\begin{remark}
Consider the 9-dimensional Dold manifold $D_{(5, 2)}$,
which is the $(\Z/2)$-quotient of $S^5 \times \C P^2$ where a generator acts by the 
product of the antipodal map and complex conjugation \cite{d65}.
This manifold is orientable and non-spin$^{c}$-able with $w_{7}(D_{(5, 2)}) \neq 0$ by Tang \cite[Lemma 2.2]{ta96} and so
$$W_{7}(D_{(5, 2)}) \neq 0.$$
\end{remark}
\begin{remark}
Diaconescu, Moore and Witten \cite[Appendix D]{dmw02} proved that there is a spinnable $10$-manifold $N$ with $W_{7}(N)\neq0$. 
Hence there exists spinnable $n$-manifolds $N$ with $W_{7}(N) \neq 0$ 
if and only if $n \geq 10$.
\end{remark}

Since $w_7(M) = \rho_2(W_7(M))$, Theorem \ref{thm:W7} implies

\begin{corollary}\label{coro:w7}
For any closed orientable spin$^{c}$-able $9$-manifold $M$,
$w_{7}(M)=0$.  \qed
\end{corollary}

\begin{remark}
Wilson \cite[Corollary 2.7, Proposition 6.2]{wi73} proved 
that $w_{7}(N) = 0$ for any spinnable manifold $N$ with $\dim N \le 12$ 
and that there exists a spinnable manifold $N^{\prime}$ with $\dim N^{\prime} = 13$ such that $w_{7}(N^{\prime}) \neq 0$. 
\end{remark}

\subsection{Some examples} \label{ss:examples}
In this subsection we give examples showing that all of the obstructions in
Theorem \ref{thm:main} are realised by orientable 9-manifolds.
In the spinnable case, $M_0 := S^1 \times \mathbb{H} P^2$ has $w_8(M_0) \neq 0$.
If we orient $\mathbb{H} P^2$ and equip $S^1$ with the non-bounding spin structure, then
we obtain a spin manifold $(M_0, s_0)$
and if we carry out spin
surgery on $S^1 \times D^8 \subset S^1 \times \mathbb{H} P^2$, then we obtain
a new spin $9$-manifold $(M_1, s_1)$ with $H^8(M_1; \Z/2) = 0$ and so $w_8(M_1) = 0$.
However, $\pi^1(M_1, s_1) = w_4(M_1) \cup \wh \phi(M_1) \neq 0$ by Equation \eqref{eq:M_1}.

In the non-spinnable case, 
the Dold manifold $M_2 : = D_{(5,2)}$ has $W_3(M_2) \neq 0$.
Suppose that $N$ has $W_3(N) = 0$ but $w_2(N) \neq 0$,
so that $N$ is spin$^c$-able but not spinnable
and that $N$ admits a contact structure; for example we could take $N = S^1 \times \C P^4$.
Then the connected sum $M_3 : = M_0 \# N$ is similarly 
spin$^c$-able but not spinnable,
$\Omega( p_{c} ( M_0 ) ) =  [w_8( M_0 )] \neq [0]$ by \eqref{eq:o8} and Theorem \ref{thm:spin} 
and $\Omega( p_c ( N ) ) = [0]$.  
By Lemma \ref{lem:Omegap}(a),
$$\Omega( p_c ( M_3) ) = \Omega( p_{c} ( M_0 ) \oplus  p_c ( N ) ) = \Omega( p_{c} ( M_0 ) ) \oplus \Omega( p_c ( N ) ) \neq [0].$$
%

\subsection{The outline of the proof} \label{ss:outline}
We fix a homotopy equivalence from $M$ to a $CW$-complex and,
abusing notation, let $M^{(q)}$ denote the $q$-skeleton of $M$. 
While we are interested in $\wt{TM}$, the stable tangent bundle of $M$,
for clarity we consider any orientable stable bundle $\wt \xi$ over $M$.
Our proof proceeds by induction over the skeleta of $M$, so we
suppose that $\wt \xi$ admits a stable complex structure $h$ over $M^{(q-1)}$. We let
\[
\mathfrak{o}_{q}(h) \in H^{q}(M; \pi_{q-1}(SO/U))
\]
denote the obstruction to extending $h$ to the $q$-skeleton $M^{(q)}$,
\[
\mathfrak{o}_q(\wt \xi) : = \{\mathfrak{o}_q(h) \mid \text{$h$ is a stable complex structure on
$\wt{\xi}|_{M^{(q-1)}}$} \} \]
denote the set of all obstructions as $h$ varies, set $\mathfrak{o}_q(M) : = \mathfrak{o}_q(\wt{TM})$
and if $\mathfrak{o}_q(\wt \xi) = \{x\}$ is a singleton we write $\mathfrak{o}_q(\wt \xi) = x$.
For $q \le 9$, the only obstructions which do not vanish in general are $\mathfrak{o}_{3}(h) \in H^{3}(M)$, 
$\mathfrak{o}_{7}(h) \in H^{7}(M)$, $\mathfrak{o}_{8}(h) \in H^{8}(M; \Z/2)$ and 
$\mathfrak{o}_{9}(h) \in H^{9}(M; \Z/2)$. 
These obstructions were studied by Massey \cite[Theorem I and equation (6)]{ma61} and
Thomas \cite[Remark p.\,\,905 ]{th67} who results combine to give

\begin{lemma}[Massey \& Thomas]\label{lem:obs}
The first obstructions satisfy
\begin{enumerate}
\item[(a)] $\mathfrak{o}_{3}(\wt \xi) = W_{3}(\wt \xi)$, 
\item[(b)] $\mathfrak{o}_{7}(\wt \xi) = W_{7}(\wt \xi)$ and 
\item[(c)] $\mathfrak{o}_{8}(\wt \xi)$ is a coset of $Sq^{2} \bigl( \rho_{2} ( H^{6}(M) ) \bigr)$.
\end{enumerate}
\end{lemma}
\noindent
The fact that $\mathfrak{o}_3(\wt \xi) = W_3(\wt \xi)$ records the well-know fact that
stably complex bundles are spin$^c$-able {\em and we now assume that $W_3(M) = 0$.}

For spin$^c$-able $9$-manifolds,
the first main contribution of this paper is Theorem \ref{thm:W7}, which states that
$W_7(M) = 0$ and so by Lemma \ref{lem:obs}(b), $\mathfrak{o}_7(M) = 0$.
The proof that $W_7(M) = 0$ uses computations of certain singular spin$^c$-bordism groups
and relies on Theorem \ref{thm:bili}, which generalises
an identity of Landweber and Stong \cite[Proposition 1.1]{ls87} for spin-manifolds to spin$^c$-manifolds. 
%

For the obstructions over the $8$-skeleton, we use results of Thomas \cite{th67} and facts about
the Stiefel-Whitney classes of orientable 9-manifolds to show in Equation \eqref{eq:o8} that
$$\Omega(p_c(M)) = \mathfrak{o}_8(M).$$
Hence $M^\circ$ admits a stable complex structure if and only if $\Omega(p_c(M)) = 0$.
In the non-spinnable case, the computation of $\Omega(p_c (M))$ can be subtle but
there is no further obstruction of $M$ admitting a stable complex structure.
In the spinnable case we use calculations in spin bordism to prove that 
$\Omega(p_c(M)) = w_8(M)$. 
These calculations include Proposition \ref{prop:AHSS-d3}, a general result computing
a $d_3$-differential
in the Atiyah-Hirzebruch Spectral Sequence for spin bordism which is of independent interest.

In the spinnable case, there is
an additional top obstruction even if $M^\circ$ admits a stable complex structure.
%
Now the top obstruction $\mathfrak{o}_9(M) \in H^9(M; \Z/2)$ is a singleton
and we first use $K$-theoretic arguments to prove that $\mathfrak{o}_9(M)$ 
can be identified with the topological index $\pi^1(M, s)$,
where $s$ is any spin structure on $M$.  We then prove Equation \eqref{eq:sigma_and_pi}, 
$\pi^1(M, s) = \sigma_{w_4}(M)$ and hence
$$ \mathfrak{o}_9(M) = \pi^1(M, s) = \sigma_{w_4}(M).$$
This completes the journey up the skeleton of $M$ and also the outline of the proof.

\begin{remark}
Our proof of Equation \eqref{eq:sigma_and_pi} in Section~\ref{ss:sigma_and_pi} 
is purely computational, using the 9-dimensional spin- and $SU$-bordism groups.
It would be interesting to have a more conceptual proof or even
index-theoretic proof of this identity.
\end{remark}

\subsection{The $\wh \phi$-invariant} \label{ss:phi}
Since the $\wh \phi$-invariant is subtle and perhaps not so well-known, we make some remarks
about its definition and history.  For $3$-connected $M$, the $\wh \phi$-invariant first appeared
in Wall's work on the classification of $3$-connected $9$-manifolds.
However, the definition of $\wh \phi(M)$ as deferred to another paper which did not appear.
The situation was largely remedied by Stolz \cite{stol85} but there were unnecessary assumptions in 
Stolz's definition of $\wh \phi(M)$, which {\em a priori} restrict it to stably parallelisable $M$.
Stolz's main idea, when generally applied, gives rise to the $\wh \phi$-invariant defined in \cite{c20}.
We explain this now, assuming that the reader is familiar with Wu-orientations as defined in \cite{br72}.
Assume that $M$ is given a $v_6$-orientation $\bar \nu_M$.  Then $M \times S^1$ inherits a $v_6$-orientation $\bar \nu_{M \times S^1}$ from $M$ and so a {\em Brown quadratic form}  
$$ \varphi(\bar \nu_{M \times S^1}) \colon H^5(M \times S^1; \Z/2) \to \Z/4 $$
is defined \cite{br72}. 
On the image of the natural map $H^5(M \times S^1, M; \Z/2) \to H^5(M \times S^1; \Z/2)$,
the form $\varphi(\bar \nu_{M \times S^1})$ is linear and takes values in $\Z/2 \subset \Z/4$.
Since this image is isomorphic to $H^4(M; \Z/2)$, the form $\varphi(\bar \nu_{M \times S^1})$
defines a homomorphism $H^4(M; \Z/2) \to \Z/2$, equivalently the element 
$\wh \phi(M) \in H^5(M; \Z/2)$.

The $\wh \phi$-invariant is an unstable tangential invariant:
if $x \in H_5(M; \Z/2)$ is represented by an embedding $f_x \colon S^5 \to M$,
then $\wh \phi(M)(x)$ is non-zero if and only if the once-stabilised 
normal bundle of $f_x$
non-trivial in $\pi_5(BO(5)) \cong \Z/2$.
The statements that $\sigma_{w_4}(M) = w_4(M) \cup \wh \phi(M)$ is a spin bordism invariant
and that $\wh \phi(M)$ is a homotopy invariant 
are both proven in \cite{c20}.

\subsection{Remarks on the map $BSO \to B(SO/U)$} \label{ss:B(SO/U)}
In this section we make some remarks on the general problem of finding 
necessary and sufficient conditions for an orientable stable vector bundle
to admit a complex structure.

Let $\wt \xi \in \wt{KO}(X)$ be an orientable stable vector bundle over a space $X$.
Whether $\wt \xi$ admits a complex structure is independent of the
orientation we chose for $\wt \xi$, so we pick an orientation and consider the associated
classifying map of $\wt \xi$, $f_{\wt \xi} \colon X \to BSO$.  We let $SO/U$ denote the homotopy fibre
of the natural infinite loop 
map $F \colon BU \to BSO$ and consider the fibration sequence
$BSO \xra{F} BU \xra{G} B(SO/U)$, as well as the lifting problem
$$
\xymatrix{  
&
BU \ar[d]^F \\
X \ar[r]^{f_{\wt \xi}} \ar@{.>}[ur] &
BSO \ar[r]^(0.4)G &
B(SO/U).
}$$
We see that $\wt \xi$ admits a complex structure if and only if $G \circ f_{\wt \xi}$ is null-homotopic.
Hence the problem of determining when orientable stable vector bundles admit complex
structures is controlled by the map $G \colon BSO \to B(SO/U)$.

Now by real Bott periodicity \cite{b59}, the space $B(SO/U)$ is homotopy equivalent to $Spin$ and so
we regard $G$ as a map $G \colon BSO \to Spin$.
From the discussion above, we see that if the composition
$$f_{\wt \xi}^* \circ G^* \colon H^*(Spin; A) \to
H^*(X; A)$$ 
is non-zero for any co-efficient group $A$, then $\wt \xi$ does not admit a complex
structure.  Massey's results that $\mathfrak{o}_3(\wt \xi) = W_3(\wt \xi)$ and 
$\mathfrak{o}_7(\wt \xi) = W_7(\wt \xi)$
correspond to the fact that the universal characteristic classes $W_3$ and $W_7$ lie in the image of the map
$G^* \colon H^*(Spin) \to H^*(BSO)$.  The fact that $\mathfrak{o}_8(M) = \Omega(p_c(M))$
is determined by a secondary cohomology operation is surely related to 
facts about $k$-invariant $\kappa \in H^8(P_7(Spin); \Z/2)$ 
of the Postnikov tower of $Spin$, where $P_7(Spin)$ denotes the $7$th Postnikov stage of the space $Spin$.  However, we will not pursue this point of view here.

\subsection{A remark on $9$-dimensional contact topology}
The overtwisted-contact structures constructed in \cite{bem15} obey an $h$-principle and may
therefore be regarded more as topological, as opposed to geometric, structures.
The existence of tight contact structures on 9-manifolds, and more generally $(2q{+}1)$-manifolds, is a difficult problem outside the scope of this paper.  
In \cite{bcs14}, Bowden, the second author
and Stipsicz gave a necessary and sufficient bordism theoretic condition for $M$ to admit a 
Stein fillable contact structure.  For example, by \cite[Theorem 5.4]{bcs14} the only 
Stein fillable homotopy 9-spheres are the standard sphere and the Kervaire 9-sphere
and the existence of tight contact structures on other homotopy 9-spheres is an open problem.
On the other hand, there are various examples of 9-manifolds with tight contact structures
which are not Stein fillable; of these we only mention those in \cite{mass12}.

\subsection{Organisation}
The organisation of the paper follows the outline of the proof of Theorem \ref{thm:main}
presented in Section \ref{ss:outline}.  
In Section \ref{s:W7} we prove that $\mathfrak{o}_7(M) = W_7(M) = 0$.
In Section \ref{s:Mcirc} we consider $\mathfrak{o}_8(M)$ 
and in Section \ref{s:top} we consider $\mathfrak{o}_9(M)$.


\vskip 0.2cm
\noindent
{\bf Acknowledgements:} The authors gratefully acknowledge the support of the China Scholarship 
Council, sponsor for the visit of the second author to the University of Melbourne in 2018.
We would also like to thank Jonathan Bowden for helpful comments.

\section{$W_7$ vanishes for spin$^c$ $9$-manifolds}  \label{s:W7}

In this section we prove Theorem \ref{thm:W7}, which states that $W_7(M) = 0$ if $W_3(M) = 0$.
For this we need the following two results,
where for a $CW$-complex $X$, $TH^*(X)$ denotes the torsion subgroup of $H^*(X)$.
\begin{lemma}\label{lem:W7}
For a closed 
orientable $n$-manifold $N$, the following three statements are equivalent:
\begin{enumerate}
\item[(a)] $W_{7}(N) = 0$;
\item[(b)] there exists an integral class $v \in H^{6}(N)$ such that $\rho_{2}(v) = w_{6}(N)$;
\item[(c)] $t \cdot w_{6}(N) = 0$ holds for any torsion class $t \in TH^{n-6}(N)$.
\end{enumerate}
\end{lemma}
\begin{proof}
All three statements can be deduced easily from 
the Bockstein sequence \eqref{eq:bseq} and the fact that $\rho_{2}( TH^{n-6}(N) )$ is the annihilator of $\rho_{2} ( H^{6}(N) )$; cf.\ \cite[Lemma 1]{ma62}.
\end{proof}

\begin{theorem}\label{thm:bili}
Let $N$ be a closed oriented $10$-dimensional spin$^{c}$-able manifold. Then for any torsion class $t \in TH^{4}(N)$, 
\begin{equation*}
\langle \rho_{2} (t) \cdot \Sq^{2} ( \rho_{2} (t) ), [N] \rangle = \langle \rho_{2}(t) \cdot w_{6}(N), ~[N] \rangle,
\end{equation*}
where $\langle ~\cdot~, ~\cdot ~\rangle$ is the Kronecker product.
\end{theorem}

The proof of Theorem \ref{thm:bili} takes the bulk of this section and is presented
after the proof of Theorem \ref{thm:W7} below.

\begin{remark}
For a closed oriented spinnable $(8k{+}2)$-dimensional manifold $N$, 
Landweber and Stong \cite[Proposition 1.1]{ls87} proved that for any $z \in H^{4k}(N)$, 
\[ \langle \rho_{2}(z) \cdot \Sq^{2} ( \rho_{2} (z) ), [N] \rangle = \langle \rho_{2}(z) \cdot \Sq^{2} ( v_{4k}(N) ), [N] \rangle, \]
where $v_{4k}(N)$ is the $4k$-th Wu class of $N$. 
In fact, Theorem \ref{thm:bili} can be generalised to spin$^c$-able $(8k{+}2)$-manifolds for 
all $k$
and we plan to demonstrate this in future work.
\end{remark}

In order to apply Lemma \ref{lem:W7} and Theorem \ref{thm:bili} to prove Theorem \ref{thm:W7} 
we need some further notation and terminology.
For any $CW$-complex $X$ and any coefficient group $G$, denote by $\Sigma X$ the suspension of $X$ and   
$$\sigma \colon H^{\ast}(X; G) \rightarrow H^{\ast +1}(\Sigma X; G)$$
the suspension isomorphism. Let
$$i_{n} \in H^{n}(K(n, \Z))~ (\text{resp.}~ i_{n}^{T} \in H^{n}(K(\Q/\Z, n); \Q/\Z))$$
be the fundamental class of the Eilenberg-MacLane space $K(n, \Z)$ (resp.\ $K(\Q/\Z, n)$) and 
$$\bar{\beta} \colon K(\Q/\Z, n) \rightarrow K(\Z, n{+}1)$$
be the Bockstein map corresponding to the Bockstein homomorphism 
\[ \bocqz \colon H^{n}(K(\Q/\Z, n); \Q/\Z) \rightarrow H^{n+1}(K(\Q/\Z, n)) \]
assocciated to the short exact sequence $\Z \to \Q \to \Q/\Z$.
By definition, we have
\begin{equation}\label{eq:beta}
\bar{\beta}^{\ast} (i_{n+1}) = \bocqz (i_{n}^{T}).
\end{equation}
For any $t \in H^{n+1}(X)$ and $z \in H^{n}(X; \Q/\Z)$,
denote by 
$$f_{t} \colon X \rightarrow K(\Z, n{+}1) ~ (\text{resp.} f_{z} \colon X \rightarrow K(\Q/\Z, n))$$ 
the map such that 
$$f_{t}^{\ast} (i_{n+1}) = t ~(\text{resp}.~ f_{z}^{\ast} (i_{n}^{T}) = z).$$
Suppose that $\bocqz (z) = t$.
Then by the definition of $\bar{\beta}$, we have 
\begin{equation}\label{eq:betaf}
f_{t} = \bar{\beta} \circ f_{z}.
\end{equation}
Finally, let $\psi \colon \Sigma^{k}K(\Z, n) \rightarrow K(\Z, n{+}k)$ denote
the map for which $\psi^{\ast} (i_{n+k})= \sigma^{k} (i_{n})$, where $\sigma^{k}$ means the $k$-fold composition of $\sigma$.

\begin{proof}[Proof of Theorem \ref{thm:W7}]
Let $\Omega^{Spin^c}_*(X)$ denote the spin$^c$-bordism groups of a space $X$.
An element of $[N, f] \in \Omega^{Spin^c}_n(X)$ is represented by a map $f \colon N \to X$
from a closed spin$^c$ $n$-manifold $N$.
We define the following homomorphisms
%
$$ \begin{array}{ll}
\cdot w_{6} \colon \Omega_{9}^{Spin^{c}}(K(\Z, 3)) \rightarrow \Z/2, &
\qquad [N, f] \mapsto \langle f^{\ast}(i_{3}) \cdot w_{6}(N), [N] \rangle,\\
\cdot w_{6} \colon \Omega_{10}^{Spin^{c}}(\Sigma K(\Z, 3)) \rightarrow \Z/2, &
\qquad [N, f] \mapsto \langle f^{\ast} ( \sigma (i_{3}) ) \cdot w_{6}(N), [N] \rangle,\\
\cdot w_{6} \colon \Omega_{10}^{Spin^{c}}(K(\Z, 4)) \rightarrow \Z/2, &
\qquad [N, f] \mapsto \langle f^{\ast} (i_{4}) \cdot w_{6}(N), [N] \rangle,
\end{array} $$
%
which fit into the following commutative diagram:
\begin{equation*}
\begin{split}
\xymatrix{
& \Omega_{9}^{Spin^{c}}(K(\Q/\Z, 2)) \ar[r]^-{\Sigma} \ar[d]_{\bar{\beta}_{\ast}} & \Omega_{10}^{Spin^{c}}(\Sigma K(\Q/\Z, 2))  \ar[d]_{\Sigma\bar{\beta}_{\ast}} & \\
& \Omega_{9}^{Spin^{c}}(K(\Z, 3)) \ar[r]^-{\Sigma} \ar[d]_{\cdot w_{6}} & \Omega_{10}^{Spin^{c}}(\Sigma K(\Z, 3)) \ar[r]^{\psi_{\ast}} \ar[d]_{\cdot w_{6}} & \Omega_{10}^{Spin^{c}}(K(\Z, 4)) \ar[d]_{\cdot w_{6}}\\
& \Z/2 \ar@{=}[r] & \Z/2 \ar@{=}[r] & \Z/2
}
\end{split}
\end{equation*}
For any $[N, f] \in \Omega_{10}^{Spin^{c}}(\Sigma K(\Q/\Z, 2))$,
\begin{align*}
 \cdot w_{6} \circ \psi_{\ast} \circ \Sigma \bar{\beta}_{\ast}  ([N, f]) = \cdot w_{6} ([N, \psi \circ \Sigma \bar{\beta} \circ f]) = \langle  f^{\ast} \circ \Sigma \bar{\beta}^{\ast} \circ \psi^{\ast} (i_{4}) \cdot w_{6}(N), [N] \rangle.
\end{align*}
Set 
\[ t := \Sigma \bar{\beta}^{\ast} \circ \psi^{\ast} (i_{4}) \in H^{4}(\Sigma K(\Q/\Z, 2)). \]
Obviously $t$ is a torsion class, hence $f^{\ast} (t) \in TH^{4}(N)$.
Since the cup product on
$\wt H^{*}(\Sigma K(\Q/\Z, 2))$ is trivial, it follows from Theorem \ref{thm:bili} that  
\begin{equation*}
\cdot w_{6} \circ \psi_{\ast} \circ \Sigma \bar{\beta}_{\ast} ([N, f]) = \langle f^{\ast} (t) \cdot w_{6}(N), [N] \rangle = \langle f^{\ast} ( t \cdot \Sq^{2} ( t ) ), [N] \rangle = 0.
\end{equation*}
This means that the homomorphism $\cdot w_{6} \circ \psi_{\ast} \circ \Sigma \bar{\beta}_{\ast}$ is trivial. Thus
\begin{equation}\label{eq:w6beta}
\cdot w_{6} \circ \bar{\beta}_{\ast}  = 0.
\end{equation}

Now every $t \in TH^{3}(M)$ satisfies 
$\bocqz (z) = t$ for some
$z \in H^{2}(M; \Q/\Z)$.
Hence $\bar{\beta} \circ f_{z} = f_{t}$ by Equation \eqref{eq:betaf}. 
Then by Equation \eqref{eq:w6beta} we have
\begin{align*}
\langle t \cdot w_{6}(M), [M] \rangle =  \cdot w_{6} ([M, f_{t}]) = \cdot w_{6} ([M, \bar{\beta} \circ f_{z}]) =  \cdot w_{6} \circ \bar{\beta}_{\ast} ([M, f_{z}]) = 0.
\end{align*}
By Lemma \ref{lem:W7}, $W_{7}(M) =  0$.
\end{proof}


The remainder of this section is devoted to the proof of Theorem \ref{thm:bili}.
For large $r$, we consider the following two cofibrations
\begin{align}
\Sigma^{r-4}K(\Q/\Z, 3) &\xrightarrow{\bar{\psi}} K(\Z,r) \xrightarrow{\bar{h}} Y_{r}, \label{eq:cofyr} \\
\Sigma^{r-4}K(\Z, 4) &\xrightarrow{\psi} K(\Z,r) \xrightarrow{h} X_{r},\label{eq:cofxr}
\end{align}
where $\bar{\psi} := \psi \circ \Sigma^{r-4} \bar{\beta}$ is the indicated composition.
By construction, there is a map 
$$\phi \colon Y_{r} \rightarrow X_{r}$$
such that the cofibrations \eqref{eq:cofxr} and \eqref{eq:cofyr} fit into the commutative diagram:
\begin{equation}\label{eq:cof}
\begin{split}
\xymatrix{
\Sigma^{r-4}K(\Q/\Z, 3) \ar[r]^-{\bar{\psi}} \ar[d]_{\Sigma^{r-4} \bar{\beta}} & K(\Z, r) \ar@{=}[d] \ar[r]^-{\bar{h}}  & Y_{r}  \ar[d]_{\phi} \\
\Sigma^{r-4}K(\Z, 4) \ar[r]^-{\psi}  & K(\Z, r) \ar[r]^-{h}  & X_{r}. 
}
\end{split}
\end{equation}
Therefore, we have the exact ladder of cohomology groups with coefficients $G=\Z$ or $\Z/2$,
\begin{equation}\label{eq:hxryr}
\begin{split}
\xymatrix{
\cdots \ar[r]^-{\delta} & H^{\ast}(X_{r}; G) \ar[r]^-{h^{\ast}} \ar[d]^{\phi^{\ast}} & H^{\ast}(K(\Z, r); G) \ar[r]^-{\psi^{\ast}} \ar@{=}[d] & H^{\ast}(\Sigma^{r-4}K(\Z, 4); G) \ar[r]^-{\delta} \ar[d]_{\Sigma^{r-4} \bar{\beta}^{\ast}} & \cdots \\
\cdots \ar[r]^-{\delta} & H^{\ast}(Y_{r}; G) \ar[r]^-{\bar{h}^{\ast}} & H^{\ast}(K(\Z, r); G) \ar[r]^-{\bar{\psi}^{\ast}} & H^{\ast}(\Sigma^{r-4}K(\Q/\Z, 3); G) \ar[r]^-{\delta} & \cdots~,
}
\end{split}
\end{equation}
where the top and bottom lines are the cohomology long exact sequences of
the cofibrations \eqref{eq:cofyr} and \eqref{eq:cofxr} respectively.

By analysing the behavior of the homomorphism $\psi^{\ast}$ when $G=\Z/2$ 
(cf.\ 
\cite[pp.\,\,627-628]{ls87}), one sees that 
\begin{align}
H^{r+5}(X_{r}; \Z/2) & \cong \Z/2  &&\text{~with generator~} \overline{\Sq^{5} ( i_{r} )}, \label{eq:hxr5}\\
H^{r+7}(X_{r}; \Z/2) & \cong \Z/2\oplus \Z/2 &&\text{~with generators~} \overline{\Sq^{7} ( i_{r} ) } \text{~and~} \delta \circ \sigma^{r-4} ( i_{4} \cdot \Sq^{2}  ( i_{4} ) ), \label{eq:hxr7}
\end{align}
where $\overline{x} \in H^{\ast}(X_{r}; \Z/2)$ denotes a class 
such that $h^{\ast}(\overline{x})=x \in H^{\ast}(K(\Z, r); \Z/2)$.
Moreover, Landweber and Stong \cite[p.\,\,628 Claim]{ls87} proved the following

\begin{lemma}\label{lem:sq5}
The generators above satisfy
\begin{equation*}
\pushQED{\qed}
\Sq^{2} \left( \overline{\Sq^{5} ( i_{r} )} \right) = \delta \circ \sigma^{r-4} ( i_{4} \cdot \Sq^{2} ( i_{4} ) ).  \qedhere
\popQED
\end{equation*}
\end{lemma}

By analyzing the cohomology groups of $X_{r}$ and $Y_{r}$ we obtain

\begin{lemma}\label{lem:tor}
The group $H^{r+5}(Y_{r})$ is torsion and there is a class $t_{Y} \in H^{r+5}(Y_{r})$ such that 
$$\rho_{2} (t_{Y}) = \phi^{\ast} \left( \overline{ \Sq^{5} ( i_{r} ) } \right).$$
\end{lemma}

\begin{proof}
Recall that the cohomology groups $H^{r+4}(\Sigma^{r-4}K(\Q/\Z, 3))$ and $H^{r+5}(K(\Z, r))$ are torsion groups. Hence from the bottom line of the exact ladder \eqref{eq:hxryr} when $G = \Z$,
we can easily deduce that $H^{r+5}(Y_{r})$ is a torsion group.

By the construction of $X_{r}$ and the Freudenthal suspension theorem we see that 
$X_{r}$ is $(r{+}4)$-connected. 
Therefore, the universal coefficient theorem implies that $H^{r+5}(X_{r})$ is torsion free. Note that 
for $i=4,5,$
$H^{r+i}(K(\Z,r))$ is a torsion group and we have
\begin{equation*}
H^{r+4}(\Sigma^{r-4}K(\Z,4))/\mathrm{torsion} \cong H^{8}(K(\Z,4))/\mathrm{torsion} \cong \Z.
\end{equation*}
Then it follows from the top line of the exact ladder \eqref{eq:hxryr} when $G=\Z$ that
\begin{equation}\label{eq:rp5}
H^{r+5}(X_{r}) \cong \Z.
\end{equation}
Now the Bockstein sequence \eqref{eq:bseq} implies that there must exists a class $x \in H^{r+5}(X_{r})$ such that $\rho_{2} (x) = \overline{ \Sq^{5} ( i_{r} ) }$. Set 
$$t_{Y} := \phi^{\ast} (x) \in H^{r+5}(Y_{r}).$$
Then $t_{Y}$ is a torsion class and 
\begin{equation*}
\rho_{2} (t_{Y}) = \rho_{2} ( \phi^{\ast} (x) ) = \phi^{\ast} ( \rho_{2} (x) ) = \phi^{\ast} 
 \left( \overline{ \Sq^{5} ( i_{r} ) } \right)\!.
\end{equation*}
\end{proof}

\begin{lemma}\label{lem:theta}
There is a class $\theta \in H^{6}(BSpin^{c};\Z/2)$, such that for any spin$^{c}$ $10$-dimensional manifold $N$ and any torsion class $t\in H^{4}(N)$, we have
\begin{equation*}
\langle \rho_{2} (t) \cdot \Sq^{2} ( \rho_{2}(t) ),\,[N] \rangle = 
\langle \tau^{\ast} (\theta) \cdot \rho_{2} (t), \,[N] \rangle,
\end{equation*}
where $\tau\colon N\rightarrow BSpin^{c}$ classifies the stable spin$^c$ tangent bundle of $N$.
\end{lemma}

\begin{proof}
We note that for large $r$ and $s$, we have 
\begin{align*}
\wt{\Omega}_{r+i}^{Spin^{c}}(K(\Z, r)) &= \pi_{2s+r+i}(MSpin^{c}(2s) \wedge K(\Z, r)) = \wt{H}_{2s+i}(MSpin^{c}(2s)) \\
&= H_{i}(BSpin^{c}(2s)) = H_{i}(BSpin^{c}).
\end{align*}
In addition, we have suspension isomorphisms
\begin{align*}
 \wt{\Omega}_{r+6}^{Spin^{c}}(\Sigma^{r-4}K(\Q/\Z, 3)) &\cong \wt{\Omega}^{Spin^{c}}_{10}(K(\Q/\Z,3)),\\
  \wt{\Omega}_{r+6}^{Spin^{c}}(\Sigma^{r-4}K(\Z, 4)) &\cong \wt{\Omega}^{Spin^{c}}_{10}(K(\Z, 4)).
\end{align*}
Then from the commutative diagram \eqref{eq:cof} of the cofibrations \eqref{eq:cofxr} and \eqref{eq:cofyr}, we have the following exact ladder,
\begin{equation*}
\begin{split}
\xymatrix{
 & \wt{\Omega}_{r+7}^{Spin^{c}}(Y_{r}) \ar[r]^-{\partial} \ar[d]_{\phi_{\ast}}  & \wt{\Omega}_{10}^{Spin^{c}}(K(\Q/\Z, 3)) \ar[r]^-{\bar{\psi}_{\ast}} \ar[d]_{\bar{\beta}_{\ast}} & H_{6}(BSpin^{c}) \ar@{=}[d]\\
& \wt{\Omega}_{r+7}^{Spin^{c}}(X_{r}) \ar[r]^-{\partial}  & \wt{\Omega}_{10}^{Spin^{c}}(K(\Z, 4)) \ar[r]^-{\psi_{\ast}}  & H_{6}(BSpin^{c}),
}
\end{split}
\end{equation*}
where $\psi_{\ast}$ 
is given by 
\begin{equation}\label{eq:psi}
\psi_{\ast}([N,f])=\tau_{\ast}([N]\cap f^{\ast} (i_{4})),
\end{equation}
for any bordism class $[N, f]\in  \wt{\Omega}^{Spin^{c}}_{10}(K(\Z, 4))$.

Denote by $\varphi \colon \wt{\Omega}^{Spin^{c}}_{10}(K(\Z,4)) \rightarrow \Z/2$ 
the homomorphism 
\begin{equation*}
\varphi([N, f])=\langle f^{\ast}(i_{4})\cdot \Sq^{2} ( f^{\ast}(i_{4}) ), ~[N] \rangle.
\end{equation*}
We claim that 
\begin{equation}\label{eq:varphi0}
\varphi \circ \bar{\beta}_{\ast} \circ \partial = 0
\end{equation}
and give the proof later.
Assuming the claim, 
since $\wt{\Omega}_{10}^{Spin^{c}}(K(\Q/\Z, 3))$ is a torsion group and all non-trivial 
torsion in $H_{6}(BSpin^{c})$ has order $2$, it follows that 
there is a homomorphism $\theta \colon H_{6}(BSpin^{c}) \rightarrow \Z/2$, or equivalently a class 
$$\theta \in \mathrm{Hom}(H_{6}(BSpin^{c}), \Z/2) \subset H^{6}(BSpin^{c}; \Z/2)$$
such that 
\begin{equation}\label{eq:theta}
\theta \circ \bar{\psi}_{\ast} = \theta \circ \psi_{\ast} \circ \bar{\beta}_{\ast} = \varphi \circ \bar{\beta}_{\ast}.
\end{equation}


Now, for any $10$-dimensional manifold $N$ and any torsion class $t \in H^{4}(N)$, there must exist an element $z \in H^{3}(N; \Q/\Z)$ such that 
$\bocqz (z) = t$.
Therefore, $$f_{t} = \bar{\beta} \circ f_{z}$$
by Equation \eqref{eq:betaf}.
Hence, on the one hand we have $ [N, f_{t}] =  \bar{\beta}_{\ast} ([N, f_{z}]) $ and 
$$\theta \circ \psi_{\ast} ([N, f_{t}]) =  \theta \circ \psi_{\ast} \circ \bar{\beta}_{\ast} ([N, f_{z}]) = \varphi \circ \bar{\beta}_{\ast} ([N, f_{z}]) = \varphi ([N, f_{t}]) = \langle \rho_{2} (t)  \cdot \Sq^{2} ( \rho_{2} (t) ), [N] \rangle$$
by Equation \eqref{eq:theta}. On the other hand we have
$$\theta \circ \psi_{\ast} ([N, f_{t}]) = \theta ( \tau_{\ast} ([N] \cap t) ) = \langle \tau^{\ast}(\theta), [N]\cap \rho_{2} (t) \rangle = \langle \tau^{\ast}\theta \cdot \rho_{2} (t), [N] \rangle$$
by the definition of $\theta$ and Equation \eqref{eq:psi}. 
This completes the proof  modulo that Equation \eqref{eq:varphi0} holds.

Now we verify Equation \eqref{eq:varphi0}.
Let
$$[(W,\partial W), (f, g)] \in \wt{\Omega}_{r+7}^{Spin^{c}}(Y_{r}) \cong \Omega_{r+7}^{Spin^{c}}(K(\Z, r), \Sigma^{r-4}K(\Q/\Z, 3))$$
be a bordism class so that $f,~g$ fit into the commutative diagram
\begin{equation*}
\begin{split}
\xymatrix{
\partial W \ar[r]^-{g} \ar@{_{(}->}[d]^{} & \Sigma^{r-4}K(\Q/\Z, 3) \ar[d]^{\bar{\psi}} \\
W \ar[r]^-{f} & K(\Z, r).
}
\end{split}
\end{equation*}
Recalling the commutative diagram \eqref{eq:cof}, it follows from the definition of $\varphi$ and Lemmas \ref{lem:sq5} and \ref{lem:tor} that
\begin{align*}
\varphi \circ \bar{\beta}_{\ast} \circ \partial ([(W,\partial W), (f, g)]) & = \langle g^{\ast} \circ (\Sigma^{r-4}\bar{\beta})^{\ast} \circ \sigma^{r-4} (i_{4} \cdot \Sq^{2}( i_{4} ) ), [\partial W] \rangle \\
&= \langle \delta \circ g^{\ast}  \circ (\Sigma^{r-4}\bar{\beta})^{\ast}  \circ \sigma^{r-4}  ( i_{4} \cdot \Sq^{2} ( i_{4} ) ), [W, \partial W] \rangle \\
&= \langle f^{\ast} \circ \phi^{\ast} \circ \delta \circ \sigma^{r-4} ( i_{4} \cdot \Sq^{2} ( i_{4} )), [W,\partial W] \rangle \\
&= \langle f^{\ast} \circ \phi^{\ast} \circ \Sq^{2} \left( \overline{ \Sq^{5} ( i_{r} ) } \right), [W,\partial W] \rangle \\
&= \langle f^{\ast} \circ \Sq^{2} ( \rho_{2} (t_{Y}) ) , [W, \partial W] \rangle \\
&= \langle  w_{2}(W) \cdot  f^{\ast} (\rho_{2} (t_{Y}) ), [W, \partial W] \rangle.
\end{align*}
Since $W$ is spin$^{c}$, there exists an element $c \in H^{2}(M)$ such that $\rho_{2} (c) = w_{2}(W)$. Recall that $t_{Y}$ is a torsion element. Then $c \cdot f^{\ast} (t_{Y})$ is a torsion element in $H^{r+7}(M, \partial W) \cong \Z$ and hence it is zero. Therefore
\begin{align*}
\varphi \circ \bar{\beta}_{\ast} \circ \partial ([(W,\partial W), (f, g)]) & = \langle  w_{2}(W)\cdot f^{\ast} ( \rho_{2} (t_{Y}) ), [W, \partial W] \rangle \\
&=\langle \rho_{2}(c \cdot f^{\ast} (t_{Y}) ), [W, \partial W] \rangle \\
&=0
\end{align*}
and this completes the proof of the claim.
\end{proof}

\begin{lemma}\label{lem:thetaz}
A class $\theta$ as in Lemma \ref{lem:theta} is well defined in 
\begin{equation*}
H^{6}(BSpin^{c};\Z/2)/\rho_{2} (H^{6}(BSpin^{c})).
\end{equation*}
\end{lemma}

\begin{proof}
For any $x\in H^{6}(BSpin^{c})$ and any torsion element $t \in H^{4}(N)$, $\tau^{\ast} (x) \cdot t$ is a torsion class in $H^{10}(N) \cong \Z$. Therefore $\tau^{\ast} (x) \cdot t = 0$ and so
\begin{align*}
\tau^{\ast}(\theta+\rho_{2} (x) ) \cdot \rho_{2} (t) &= \tau^{\ast}(\theta) \cdot \rho_{2} (t) + \rho_{2} (\tau^{\ast} (x) ) \cdot \rho_{2} (t) \\
&=  \rho_{2} (t) \cdot \Sq^{2} ( \rho_{2} (t) ) + \rho_{2}(\tau^{\ast} (x) \cdot t)\\
&= \rho_{2} (t) \cdot \Sq^{2} ( \rho_{2} (t) ).
\end{align*}
Thus, the class $\theta+\rho_{2} (x)$ has the same property as $\theta$.
\end{proof}

Now classical computations show that $H^{6}(BSpin^{c};\Z/2)/\rho_{2} ( H^{6}(BSpin^{c}) ) \cong \Z/2$,
generated by $[w_6]$ and we have

\begin{lemma}\label{lem:theta0}
For a class $\theta$ as in Lemma \ref{lem:theta}, 
$[\theta] \in H^{6}(BSpin^{c};\Z/2)/\rho_{2} ( H^{6}(BSpin^{c}) )$ is a generator.
\end{lemma}

\begin{proof}
It follows from Diaconescu, Moore and Witten \cite[Appendix D]{dmw02} that there exists a $10$-dimensional spin manifold $N$ such that $W_{7}(N) \neq 0$. Then Lemma \ref{lem:W7} implies that there must exist a torsion class  $t \in H^{4}(N)$ such that 
$\rho_{2} (t) \cdot w_{6}(N) = \rho_{2} (t) \cdot \Sq^{2} ( v_{4}(N) ) \neq 0$. 
Hence $\rho_{2} (t) \cdot \Sq^{2} ( \rho_{2} (t) ) = \rho_{2}(t) \cdot \Sq^{2} ( v_{4}(N) )\neq 0$ by Landweber and Stong \cite[Proposition 1.1]{ls87}. This completes the proof.
\end{proof}

\begin{proof}[Proof of Theorem \ref{thm:bili}]
Since $H^{6}(BSpin^{c};\Z/2)/\rho_{2} (H^{6}(BSpin^{c}) ) \cong \Z/2$ with generator 
$[w_{6}]$, Theorem \ref{thm:bili} follows easily from Lemmas \ref{lem:theta}, \ref{lem:thetaz} and \ref{lem:theta0}.
\end{proof}


\section{The existence of stable complex structures on $M^\circ$} \label{s:Mcirc}
Throughout this section, $M$ will be an orientable $9$-manifold. 
Denote by $\mdot := M - \mathrm{int}(D^9)$ the space obtained from $M$ by removing the interior of a small
embedded $9$-disc.
Applying Morse theory, we equip $M$ with a $CW$-structure in which the $8$-skeleton of $M$
is a deformation retract of $M^\circ$.
In Section \ref{s:W7} we showed that if $W_3(M) = 0$ then $M$ admits a
stable complex structure of its $7$-skeleton.
In this section, we describe $\mathfrak{o}_8(M)$, 
the obstruction for $M$ to admit a stable complex structure over $\mdot$ 
in general in terms of the secondary cohomology operation $\Omega$ from the Introduction.
We also describe $\mathfrak{o}_8(M)$ terms of the characteristic classes and the cohomology ring of 
$M$ in some special cases, such as when $M$ is spin.

\subsection{Preliminaries} \label{ss:prelim}
Recall that $TM$ denotes the tangent bundle of $M$.
In this subsection we establish some preliminary results about
stable complex structures on $TM|_{M^{(7)}}$, complex vector bundles over $\mdot$, 
the Stiefel-Whitney classes of $M$, the action of Steenrod squares on the mod~$2$ cohomology groups of $M$ and the secondary cohomology operation $\Omega$.


Firstly, we consider stable complex structures on $TM|_{M^{(7)}}$ and complex vector bundles 
over $\mdot$.  From Lemma \ref{lem:obs} and Theorem \ref{thm:W7} we immediately obtain

\begin{lemma}\label{lem:M7}
A $9$-dimensional manifold $M$ admits a stable complex structure over $M^{(7)}$ if and only if $M$ is spin$^{c}$-able.  \hfill \qed
\end{lemma}



\begin{lemma}
\label{lem:acsc1}
Let $M$ be a spin$^c$-able $9$-manifold. 
For any class $c \in H^{2}(M)$ with $\rho_2 ( c ) = w_2 (M)$, there exists a stable complex structure $\eta$ of $TM|_{ M^{(7)} }$ such that $c_1(\eta) = c$.
\end{lemma}

\begin{proof}
For any $x \in H^{2}(M)$, denote by $l_{x}$ the complex line bundle over $M$ with $c_{1}(l_{x}) = x$.
Suppose that $\eta^{\prime}$ is a stable complex structure of $TM|_{ M^{(7)} }$ by Lemma \ref{lem:M7}.
Then $\rho_2 ( c_1(\eta^{\prime}) ) = w_2(M)$ and
hence $c = c_1( \eta^{\prime} ) + 2x$ for some $x \in H^{2} ( M )$.
Therefore, it follows that $\eta = \eta^{\prime} + l_x - l_{ - x} $ is a stable complex structure of $TM|_{ M^{(7)} }$ with $c_1( \eta ) = c$.
\end{proof}

\begin{lemma}\label{lem:8ext9}
In the stable range, every complex vector bundle $\eta$ over $\mdot$ can be extended to 
a complex vector bundle over $M$. 
\end{lemma}

\begin{proof}
By definition, 
$M$ is homeomorphic to $\mdot \cup_{f} D^{9}$ for some map $f\colon S^{8} \rightarrow \mdot$, which attaches the $9$-disc $D^{9}$. 
In order to prove that $\eta$ can be extended to $M$, it suffices to prove that $f^{\ast}(\eta)$ is stable trivial. 
Hence it only need to prove that the fourth Chern class $c_{4}(f^{\ast}(\eta)) = 0$.
Since $c_{4}(f^{\ast}(\eta)) = f^{\ast}(c_{4}(\eta))$ and $f^{\ast}\colon H^{8}(\mdot) \rightarrow H^{8}(S^{8})$ is trivial, we are done.
\end{proof}

Secondly, 
we consider 
the Stiefel-Whitney classes of an orientable $9$-manifold $M$ 
and the action of Steenrod squares on the mod~$2$ cohomology groups of $M$.
%
Denote by $v_{i}(M)\in H^{i}(M;\Z/2)$ the $i$-th Wu-class of $M$.
It is known that $v_{1}(M) = w_{1}(M) = 0$
and it follows from the definition of $v_{i}(M)$ and the Steenrod relation $\Sq^{3} = \Sq^{1}  \Sq^{2}$ 
that $v_{i}(M)=0$ for $i\neq2,~4$; see \cite[\S 11]{ms74b}.
Therefore, combining these facts with Wu's formula 
$w_{k}(M)=\sum_{i=0}^{k}\Sq^{i} ( v_{k-i}(M) )$ (cf. \cite[p. 132]{ms74b}),
we can deduce that
\begin{align}
v_{2}(M) & = w_{2}(M), \label{eq:v2}\\
v_{4}(M) & = w_{4}(M)+w_{2}^{2}(M), \label{eq:v4} \\
w_6(M) & = \Sq^2 ( v_4(M) ), \label{eq:w6} \\
w_8(M) & = w_4^2(M) + w_2^4(M), \label{eq:w8} \\
w_{2i+1}(M) & = \Sq^1 ( w_{2i}(M) ), \text{~ for ~} i = 1,~ 2, ~ 3, \label{eq:wodd}
\end{align}
and 
\begin{equation} \label{eq:w9}
w_9(M) = 0. 
\end{equation}

The above identities allow us to prove

\begin{proposition}\label{prop:sw}
If $M$ is a spin$^c$-able $9$-manifold, then
the Stiefel-Whitney classes of $M$ satisfy the following relations:
 $w_{k}(M)  =0$ for odd $k$,  $w_{6}(M)  = \Sq^{2} ( w_{4}(M) )$ and
$$w_{2}(M)w_{4}(M)   = w_{2}(M)w_{6}(M) = 0.$$
\end{proposition}

\begin{proof}
Note that $M$ being spin$^{c}$-able implies that 
$w_3(M) = \Sq^{1} ( w_{2}(M) ) = 0$ and $ w_5(M) = \Sq^{1} ( w_{4}(M) ) = 0$. 
Then combining these facts with Corollary \ref{coro:w7} and Equation \eqref{eq:w9}, we can deduced that $w_k(M) = 0$ for odd $k$.
Moreover, since $w_3(M) = \Sq^{1} ( w_{2}(M) ) = 0$, it follows from Equations \eqref{eq:v4}, \eqref{eq:w6} and the Cartan formula that
\begin{align*}
w_{6}(M) = \Sq^{2} ( v_{4}(M) ) =  \Sq^{2} ( w_{4}(M) ).
\end{align*}
Now, the fact $\Sq^{2} ( w_{4}(M) ) = w_{2}(M) w_{4}(M) + w_{6}(M)$ by 
Wu's explicit formula
(cf. \cite[Problem 8-A]{ms74b}) implies that $w_{2}(M) w_{4}(M)=0$ and 
\[
w_{2}(M)w_{6}(M) = w_{2}(M)\Sq^{2} ( w_{4}(M) ) = \Sq^{2} (w_{2}(M)w_{4}(M)) + w^{2}_{2}(M)w_{4}(M) = 0.
\]
This completes the proof.
\end{proof}

We next consider the action of the Steenrod squares on $H^*(M; \Z/2)$.
Recall
the set
\begin{equation*} 
\mathcal{D}_{M} := \{x\in H^{1}(M; \Z/2) ~|~ x \cdot w_{2}(M) \in \rho_{2}(TH^{3}(M))\},
\end{equation*}
where $TH^3(M)$ is the torsion subgroup of $H^3(M)$.

\begin{lemma}\label{lem:dm}
The set
$\mathcal{D}_{M}$ is the annihilator of $\Sq^{2} ( \rho_{2}(H^{6}(M)) )$ with respect to the cup product .
\end{lemma}

\begin{proof}
For any $x \in \mathcal{D}_{M}$ and $y \in H^{6}(M)$, it follows from the Cartan formula, the definition of Wu-classes and Equation \eqref{eq:v2} that
\begin{equation*}
x \cdot \Sq^{2} ( \rho_{2} (y) )= \Sq^{2}(x \cdot \rho_{2}(y) )= v_2(M) \cdot x \cdot y = w_{2}(M) \cdot x \cdot y. 
\end{equation*}
Since the annihilator of $H^{6}(M)$ is $TH^{3}(M)$ (cf.\ \cite[Lemma 1]{ma62}), 
the lemma can be deduced easily from the definition of $\mathcal{D}_{M}$.
\end{proof}

\begin{lemma}\label{lem:sq}
If $M$ is a orientable $9$-manifold, then
\begin{enumerate}
\item[(a)] for any $y \in H^{6}(M; \Z/2)$, $\Sq^{2} ( y ) = w_{2}(M) y,$
\item[(b)] for any $z \in H^{4}(M; \Z/2)$, $z^2 = v_4(M ) z = ( w_4(M)  + w_2^2(M) ) z,$
\item[(c)] if $\boc(x) = 0$ for any $x \in \mathcal{D}_{M}$, then
$\Sq^{1} ( H^{7}(M; \Z/2) ) \subseteq \Sq^{2} ( \rho_{2} (H^{6}(M)) ).$

\end{enumerate}
If, in addition, $M$ is spin$^c$-able then
\begin{enumerate}
\item[(d)] For any $u \in H^{2}(M)$, $w_{6}(M) \cdot \rho_{2}(u) \in \Sq^{2} ( \rho_{2} (H^{6}(M)) )$,
\item[(e)] if  $w_{4}(M)=0$, then $ \rho_{2}(v^{2}) \in \Sq^{2} ( \rho_{2} (H^{6}(M)) )$ for any $v \in H^{4}(M)$.
\end{enumerate}

\end{lemma}

\begin{proof}
For any $x \in H^{1}(M; \Z/2)$ and $y \in H^{6}(M; \Z/2)$, since $v_1(M) = 0$ and $\Sq^1 \circ \Sq^1 = 0$, 
the Cartan formula gives
$$\Sq^1 ( x ) \cdot \Sq^1 ( y ) = \Sq^1 ( x \cdot \Sq^1 ( y ) ) + x \cdot \Sq^1 ( \Sq^1 ( y ) ) = v_1(M) \cdot x \cdot \Sq^1 ( y ) = 0.$$
Hence for any $z \in H^4 (M; \Z/2)$, we have $\Sq^1 ( x ) \cdot \Sq^3 ( z ) = \Sq^1 (x)  \cdot \Sq^1 ( \Sq^2 ( z ) ) = 0$.
Therefore, it follows from Equations \eqref{eq:v2} and \eqref{eq:v4} and the Cartan formula that
$$x \cdot \Sq^{2} ( y )  = \Sq^{2} (x \cdot y) = x \cdot v_{2}(M) \cdot y = x \cdot w_2(M) \cdot y,$$
and 
$$x \cdot z^2  = x \cdot \Sq^4 ( z ) = \Sq^4 ( x z )  = x \cdot v_4(M) \cdot z = x \cdot (w_4(M) + w_2^2(M) ) \cdot z.$$
Parts $(a)$ and $(b)$ now follow immediately from Poincar\'e duality.

For any $x \in \mathcal{D}_M$ and $z \in H^7(M;\Z/2)$, if $\boc(x) = 0$ then $\Sq^1 ( x ) = 0$.
Hence since $v_1(M) = 0$, the Cartan formula gives
$$ x \cdot \Sq^{1} ( y ) = \Sq^{1}(x \cdot y) + \Sq^{1}( x ) \cdot y = 0.$$
Part $(c)$ now follows from Lemma \ref{lem:dm}.

Now, suppose that $M$ is spin$^c$-able. 
For any $u \in H^2(M)$, 
it follows from Equation \eqref{eq:w6}, Parts $(a)$ and $(b)$ and the Cartan formula that
\begin{align*}
w_6(M) \rho_2 (u) & = \Sq^2 ( v_4(M) ) \cdot \rho_2 (u) \\
& = \Sq^2( v_4(M) \cdot \rho_2 (u) ) + v_4(M) \cdot \Sq^2 ( \rho_2(u) ) \\
& = v_4(M) (w_2(M) \rho_2 (u) + ( \rho_2 (u) )^2 ) \\
& = (\rho_{2}(u))^{4} + w_{2}^{2}(M) \cdot (\rho_{2}(u))^{2} \\
& = \Sq^2 (  ( \rho_2(u) )^3 + w_2(M) \cdot ( \rho_2 (u) )^2 ),
\end{align*} 
which completes the proof of Part $(d)$.

Finally, for any $x \in \mathcal{D}_{M}$ any $v \in H^{4}(M)$, Part $(b)$ and $w_4(M) = 0$ together imply that 
$$x \cdot \rho_{2}(y^{2})  =  x \cdot w^{2}_{2}(M) \cdot \rho_{2}(y) .$$ 
Then since $x \cdot w_{2}(M) \in \rho_{2}(TH^{3}(M))$ by the definition of $\mathcal{D}_M$, 
it follows that $x \cdot \rho_{2}(y^{2}) = 0 $.  Part $(e)$ now follows from Lemma \ref{lem:dm}.
\end{proof}

We conclude this subsection with
the following facts about the secondary cohomology operation $\Omega$.
\begin{lemma}[Lemma 1.5 in Thomas \cite{th67}]
\label{lem:Omegap}
Let $D_{\Omega}$ be the domain of $\Omega$ as in Section \ref{s:intro}. Then
\begin{enumerate}
\item[(a)] For any  $u, v \in D_{\Omega}$, 
$\Omega(u + v) = \Omega(u) + \Omega(v) + [\rho_2( u \cdot v )]$;

\item[(b)] For any  $w \in H^4 (M)$, $\Omega( 2w ) = [ \rho_2 (w^2) ]$. \hfill \qed
\end{enumerate}
\end{lemma}

\subsection{The obstruction $\mathfrak{o}_8(M)$}  \label{ss:o8}
In this subsection we assume that $M$ is a spin$^c$-able $9$-manifold.
We will determine the obstruction $\mathfrak{o}_8(M)$ in terms of the secondary cohomology operation $\Omega$ as in Section \ref{s:intro} and a spin$^c$ characteristic class $p_c(M)$ of $M$; see Equation \eqref{eq:o8}.
We also give a condition for $0 \in \mathfrak{o}_8(M)$ 
in terms of the Stiefel-Whitney classes of $M$ and a complex vector bundle over $M$;
see Proposition \ref{prop:8obs}.

Let $X$ be a path-connected $CW$-complex 
and consider the canonical maps $F_1 \colon BSO \to BO$ and $F_2 \colon BSpin \to BSO$.
The fibre of $F_1$ is $\Z/2$, the fibre of $F_2$ is $\R P^\infty$, infinite real projective space
and we have the following exact sequences of abelian groups
\begin{align*}
\cdots \to [X, O] \to [X, \Z/2] & \to [X, BSO] \xrightarrow{ F_{1 \ast} } [X, BO] \to \cdots, \\
\cdots \to [X, SO] \to [X, \R P^{\infty}] & \to [X, BSpin] \xrightarrow{ F_{2 \ast} } [X, BSO] \to \cdots.
\end{align*}
Since the $H$-space maps $O \to \Z/2$ and $Spin \to \R P^\infty$ are both split,
it follows that the homomorphisms $[X, O] \to [X, \Z/2]$ and $[X, SO] \to [X, \R P^{\infty}]$ 
in the exact sequences above are both surjective.
Hence we have
\begin{proposition}
\label{prop:uniqSOSpin}
The homomorphisms $F_{1 \ast} \colon [X, BSO] \to [X, BO]$ and $F_{2 \ast} \colon [X, BSpin] \to [X, BSO]$ induced by $F_1$ and $F_2$ respectively are both injective. \hfill \qed
\end{proposition}

Let $w_i$, $p_i$ and $c_i$ denote the universal Stiefel-Whitney classes, Pontrjagin classes and Chern classes respectively.  Recall from Thomas \cite[Theorem (1.2)]{th62} that 
$H^4(BSpin) \cong \Z$ is generated by $q_1$,  and $H^{8}(BSpin) \cong \Z \oplus \Z$ is generated by $q_1^2$ and $q_2$ respectively.
These generators have the following relationships with $p_i$ and $w_i$
(cf.\ \cite[(1.5), (1.6) and (1.8)]{th62}):
\begin{align}
F_2^{\ast} ( p_1 ) & = 2 q_1, & \rho_2 ( q_1 )  = F_2^{\ast} ( w_4 ), \label{eq:q1} \\
F_2^{\ast} ( p_2 ) & = 2 q_2 + q_1^2, & \rho_2 ( q_2 )  = F_2^{\ast} ( w_8 ). \label{eq:q2}
\end{align}
Now denote by $BSU$ the classifying space of the stable group $SU$
and by $F_3 \colon BSU \to BSpin$ the canonical map.
Since $F_3^{\ast} ( F_2^{\ast} ( p_1 ) ) = - 2 c_2$, $F_3^{\ast} ( F_2^{\ast} ( p_2 ) ) =  2 c_4 + c_2^2$ and $H^{\ast}(BSU)$ is torsion free, it follows that
\begin{align}\label{eq:cq}
F_3^{\ast} ( q_1 ) = - c_2, \qquad F_3^{\ast} ( q_2 ) = c_4.
\end{align}

Let $\wt \xi$ be a stable real vector bundle over $X$
and suppose that it is classified by a map $\wt \xi \colon X \to BO$.
The stable bundle $\wt \xi$ is spinnable if and only if the map $\wt \xi$ admits a lift 
$\wt \zeta \colon X \to BSpin$.
Recall that $\wt \xi$ is orientable if and only if $w_1(\wt \xi) = 0$ and that $\wt \xi$ is spinnable if and only if $w_1(\wt \xi) = 0$ 
and $w_2(\wt \xi) = 0$.
Now suppose that the stable vector bundle $\wt \xi$ is spinnable and $\wt \zeta \colon X \to BSpin$ is a lifting of the map $\wt \xi$.
Then it follows from Proposition \ref{prop:uniqSOSpin} that the map $\wt \zeta \colon X \to BSpin$ 
is unique up to homotopy and hence 
the spin characteristic classes of $\wt \xi$
\begin{align*}
q_i(\wt \xi) = \wt \zeta^{\ast} ( q_i ), ~ i = 1,~2
\end{align*}
are well defined and depend neither on the orientation nor on spin structure on $\wt \xi$.

Let $M$ be a spin$^c$-able $9$-manifold,
and $c \in H^2(M)$ a spin$^c$ characteristic class of $M$, i.e.\,\,$\rho_2(c) = w_2(M)$.
Denote by $l_{c}$ the complex line bundle over $M$ with $c_{1}(l_{c}) = c$.
Then, in the stable range, the real vector bundle $TM {-} (l_c)_{ \R }$ is spinnable, and we set
%
\begin{equation} \label{eq:pc}
p_c( M ) := - q_1( TM {-} (l_c)_{ \R } ) \in H^4 (M).%
\footnote{We note that $q_1( TM {-} (l_c)_{ \R } ) = \check{p}(M)$ and hence $p_c ( M )= - \check{p}(M)$,
where $\check{p}$ is the defined by the first author and 
Nordstr\"om in \cite[Section 2.7]{cn19}.}
%
\end{equation}
From the definition of $p_c ( M )$ and the remarks above, we see that 
$p_c ( M )$ depends neither on the orientation of $M$ 
nor on the spin structure of $TM {-} (l_c)_{ \R }$.
Hence, we have

\begin{lemma} \label{lem:pc}
Let $M$ be a spin$^c$-able $9$-manifold.  The class $p_c ( M )$ depends only on the choice of $c$.
 \hfill \qed
\end{lemma}

Note that 
$ \rho_2( q_1 ( TM {-} (l_c)_{ \R } ) ) = w_4( TM {-} (l_c)_{ \R } ) $ by \eqref{eq:q1} and $w_4( TM {-} (l_c)_{ \R } ) = w_4(M)$ by a straightforward calculation.  It follows that $\rho_2( p_c ( M ) ) = w_4(M)$. 
Now since $\Sq^2 ( w_4 (M) ) = w_6(M)$ by Proposition \ref{prop:sw} and $\boc( w_6(M) ) = W_7(M) = 0 $ by Theorem \ref{thm:W7}, it follows that 
the action of $\Omega$ is well-defined on $ p_c (M)$.

\begin{lemma} \label{lem:omegac}
Let $M$ be a spin$^c$-able $9$-manifold. 
The value of $\Omega( p_c (M) )$ depends neither on the choice of $c$ nor on 
the choice of the orientation of $M$.
\end{lemma}

\begin{proof}
By Lemma \ref{lem:pc}
we only need to prove that $\Omega( p_c(M) )$ does not depend on the choice of $c$.
Let $\eta$ be a stable complex structure of $TM|_{M^{(7)}}$ with $c_1(\eta) = c$,
which exists by Lemma \ref{lem:acsc1}. 
Then we have $c_1(\eta - l_c) = 0$ and by Equation \eqref{eq:cq} we have
\begin{align}
\label{eq:c2pc}
c_{2}(\eta) = c_2( \eta {-} l_c ) = - q_1(\eta_{ \R } {-} (l_c)_{ \R }) = - q_1( TM {-} (l_c)_{\R} ) =  p_c (M),
\end{align}
Therefore, for any $x \in H^2 (M)$, 
$\eta^{\prime} = \eta + l_{x} - l_{-x}$ is also a stable complex structure of $TM|_{ M^{(7)} }$ with $c_{1}(\eta^{\prime}) = c + 2x$, $c_{2}(\eta^{\prime}) = p_{c+2x} (M)$.
By Thomas \cite[Lemma 1.1]{th67}
$$\Omega( p_c (M) )= \Omega( p_{c+2x}(M) )$$
and so the value of $\Omega( p_c(M) )$ does not depend on the choice of $c$.
\end{proof}


We now show that obstruction $\mathfrak{o}_8(M)$ 
equals $\Omega(p_c(M))$.
By \eqref{eq:c2pc}, $p_c (M) = c_2(\eta)$ for some stable complex structure $\eta$ of $TM|_{M^{(7)}}$ 
with $c_1(\eta) = c$.  From the proof of \cite[Theorem 1.2]{th67}, one sees that Thomas (cf. \cite[p. 898, (2.6)]{th67}) has proven that
$$
\mathfrak{o}_8 (M) = \Omega(p_c(M)) + [ w_8(M) + w_4^2(M) + w_2^2 (M) w_4(M)  ].$$
By Equation \eqref{eq:w8}, 
$w_{8}(M)  = w^{2}_{4}(M) + w^{4}_{2}(M)$ and by Proposition \ref{prop:sw} 
$w_{2}(M) w_{4}(M)  = 0$ 
and by the Cartan formula $w^{4}_{2}(M) = \Sq^{2} ( w^{3}_{2}(M) ) \in \Sq^{2} ( \rho_{2}(H^{6}(M)) )$.
It follows that $[ w_8(M) + w_4^2(M) + w_2^2 (M) w_4(M)  ] = [0]$ and hence
\begin{equation}\label{eq:o8}
\mathfrak{o}_8 ( M ) = \Omega( p_c ( M ) ).
\end{equation}
Combining Equation \eqref{eq:o8} with Lemma \ref{lem:M7}, we obtain

\begin{theorem} \label{thm:Omega}
Let $M$ be a spin$^c$-able $9$-manifold. 
$M$ admits a stable complex structure over $\mdot$ if and only if 
\[
\pushQED{\qed}
\Omega ( p_c (M) ) = [0] \in H^{8}(M; \Z/2)/\Sq^{2} ( \rho_{2} (H^{6}(M)) ). \qedhere
\popQED
\] 
\end{theorem}

The majority of the remainder of Section \ref{s:Mcirc} is devoted to determining the obstruction $\mathfrak{o}_8 ( M )$ in terms of the characteristic classes and cohomology ring of $M$.
%
We conclude this subsection by proving the homotopy
invariance of $\Omega(p_c(M))$, which requires a preparatory 

\begin{lemma}
\label{lem:q1mod4}
The class $\rho_4(q_1) \in H^4(BSpin; \Z/4)$ is a fiber homotopy invariant.
\end{lemma}

\begin{proof}
Let $J_{SO} \colon BSO \to BSG$ be the map classifying the transition from oriented stable vector bundles 
to oriented stable spherical fibrations and $G/O$ be its homotopy fibre.
We consider the following commutative diagram
\[
\xymatrix{
\cdots \ar[r] & H^3(G/O; \Z/2) \ar[d]^{=} \ar[r]^{\beta^{\Z/\!2}} & H^{4}(G/O) \ar[d]^{\rho_4} \ar[r]^{\times 2} & H^4(G/O) \ar[d]^{\rho_8} \ar[r] & \cdots \\
\cdots \ar[r] & H^3(G/O; \Z/2)  \ar[r]^{\beta^{\Z/4}_{\Z/2}} & H^{4}(G/O; \Z/4)  \ar[r]^{\times 2} & H^4(G/O; \Z/8)  \ar[r] & \cdots~,
}
\]
where the bottom line is the Bockstein sequence associated to the coefficient sequence 
$\Z/4 \xrightarrow{\times 2} \Z/8 \to \Z/2$.
Now it is well-known that $H^4(G/O) \cong \Z$; this can be seen from knowledge of 
the homotopy groups $\pi_i(G/O)$ for $i \leq 5$ and the Serre Spectral Sequence
of the $4$th Postnikov stage of $G/O$.
Hence 
$\boc$ and 
$\beta^{\Z/4}_{\Z/2} = \rho_4 \circ \boc$ are both trivial and
therefore $\times 2 \colon H^{4}(G/O; \Z/4) \to H^{4}(G/O; \Z/8)$ is injective.

Denote by $BGSpin$ the homotopy fibre of $w_2 \colon BSG \to K(\Z/2, 2)$, where $w_2$ represents the generator of $H^2(BSG; \Z/2) \cong \Z/2$.
Since $BSpin$ is the homotopy fibre of $w_2 \colon BSO \to K(\Z/2, 2)$,
there is a commutative diagram of fibrations
%
\begin{equation} \label{diag:BGSpin}
\xymatrix{
G/O \ar[d]^{i_{Spin}} \ar[r]^= &
G/O \ar[d]^{i_{SO}} \\
BSpin \ar[d]^{J_{Spin}} \ar[r]^{F_2} &
BSO \ar[d]^{J_{SO}} \\
BGSpin \ar[r] &
BSG,}
\end{equation}
where $J_{Spin}$ is the pull-back of $J_{SO}$.
Since $q_1$ is additive under the Whiney sum of spin vector bundles,
it suffices to prove that $i^{\ast}_{Spin}(\rho_4(p_1)) = 0$.
By a theorem of Madsen \cite[(5.1) Theorem]{mad75}, 
$\rho_8(p_1) \in H^4(BSO; \Z/8)$ lies in the image of $J_{SO}^*$
and so
$$ i^{\ast}_{SO}(\rho_8 (p_1) ) = 0 \in H^4 (G/O; \Z/8).$$
By Equation \eqref{eq:q1} and the commutative diagram \eqref{diag:BGSpin}, we have
$$ 2 i^{\ast}_{Spin}(\rho_4 (q_1)) =  i^{\ast} _{Spin}(\rho_8( 2 q_1)) = 
i^{\ast}_{Spin}(F_2 ^{\ast}(\rho_8 (p_1) ) ) =  i^{\ast}_{SO}(\rho_8(p_1) )  = 
0 \in H^4 (G/O; \Z/8).$$
The injectivity of $\times 2 \colon H^{4}(G/O; \Z/4) \to H^{4}(G/O; \Z/8)$ implies that 
$ i^{\ast}_{Spin}(\rho_4 (q_1) ) = 0$.
%
%
\end{proof}

\begin{remark}
In fact $\rho_4(q_1)$ lies in the image of $J^{\ast}_{Spin}$, as can be seen as follows.
Since $G/O$ is simply connected and $BGSpin$ is $2$-connected, 
the Serre spectral sequence of the fibration $G/O \to BSpin \to BGSpin$
gives a short exact sequence 
$$H^{4}(BGSpin; \Z/4) \xrightarrow{J^{\ast}_{Spin}} H^4(BSpin; \Z/4) \xrightarrow{i^{\ast}_{Spin}} H^4(G/O; \Z/4).$$
Since $i^{\ast}_{Spin}( \rho_4( q_1 ) ) = 0$, we have that 
$\rho_4(q_1)$ lies in the image of $J^{\ast}_{Spin}$. 
\end{remark}

\begin{proposition}
\label{prop:htpopc}
Let $M$ be a spin$^c$-able $9$-manifold. 
Then $\Omega( p_c (M) )$ is a homotopy invariant of $M$.
\end{proposition}

\begin{proof}
Let $f  \colon M \to M^{\prime}$ be a homotopy equivalence 
and chose orientations on $M$ and $M'$ so that $f$ is orientation preserving.
Let $c^{\prime} \in H^2 ( M^{\prime} )$ be any integral lift of  $w_2( M^{\prime} )$.
Since the Stiefel-Whitney classes are homotopy invariants, it follows that $c = f^{\ast} ( c^{\prime} )$ is an integral lift of $ w_2(M)$.

By Atiyah \cite[Theorem 3.5]{at61}, the stable bundles $TM$ and $f^{\ast}(TM^{\prime})$ are fibre homotopy equivalent.
This implies that $TM - (l_c)_{\R}$ and $f^{\ast}(TM^{\prime} - ( l_{c^{\prime}} )_{\R}) = f^{\ast}(TM^{\prime}) - (l_c)_{\R} $ are fiber homotopy equivalent.  
Now by Lemma \ref{lem:q1mod4} we have that
$q_1( TM - (l_c)_{\R} ) \equiv q_1(f^{\ast}(TM^{\prime} - ( l_{c^{\prime}})_{\R}) ) \bmod{4}$
and from the definiton of $p_c(M)$ we have
$$p_c(M) \equiv f^{\ast}( p_{ c^{\prime} } ( M^{\prime} ) ) \bmod{4}.$$ 
Thus there is a class $x \in H^{4}(M)$ such that $f^{\ast}( p_{ c^{\prime} } ( M^{\prime} ) ) - p_c (M) = 4 x$ and by Lemma \ref{lem:Omegap} by
\begin{align*}
f^{\ast} ( \Omega( p_{ c^{\prime} } ( M^{\prime} ) ) ) 
=  \Omega( f^{\ast}( p_{ c^{\prime} } ( M^{\prime} ) ) )
= \Omega( p_c (M) + 4x) = \Omega( p_c (M) ) + \Omega( 4x ) 
= \Omega( p_c(M) ),
\end{align*}
which completes the proof.
\end{proof}

\subsection{An application of the differential Riemann-Roch theorem}

The main result of this subsection is Theorem \ref{thm:betad}, which we prove using
the differential Riemann-Roch theorem.

For a spin$^c$-able $9$-manifold $M$, 
let $c \in H^{2}(M)$ and $v \in H^{6}(M)$ be integral lifts of 
$w_2(M)$ and $w_6(M)$ respectively: such a $v$ exists by Theorem \ref{thm:W7}. 
By Proposition \ref{prop:sw}, $\rho_{2}(cv) = w_{2}(M)w_{6}(M) = 0$.
So there is are integral classes $d \in H^8(M) $ such that $2 d = cv$ and
we ambiguously denote such a class $d$ by $cv/2$.
Suppose that $\boc (x) =0$ for all $x \in \mathcal{D}_{M}$.
It follows from Lemma \ref{lem:sq}(a)(c)(d) of that the coset 
$[\rho_2(cv/2)]$ depends neithor on the choice of $c$ or $v$ nor on the selection of $cv/2$
and so the same holds for  the coset $[w_{8}(M) - \rho_{2}(cv/2)]$.

\begin{theorem}\label{thm:betad}
Let $M$ be a spin$^c$-able $9$-manifold. 
Suppose that $\boc (x) =0$ for all $x \in \mathcal{D}_{M}$. 
Then $M$ admits a stable complex structure over $\mdot$ if and only if 
$$[w_{8}(M) - \rho_{2}(cv/2)] = [0] \in H^{8}(M; \Z/2) / \Sq^{2} ( \rho_{2} ( H^{6}(M) ) )$$
holds for some integral lifts $c \in H^{2}(M)$ and $v \in H^{6}(M)$ of $w_2(M)$ and $w_6(M)$ respectively.
\end{theorem}

In order to prove Theorem \ref{thm:betad}, we need the two following results.

\begin{proposition}\label{prop:8obs}
Let $M$ be a $9$-dimensional manifold. 
Then $M$ admits a stable complex structure over $\mdot$ (i.e.\ $0 \in \mathfrak{o}_8(M)$)
if and only if there exists a complex vector bundle $\eta$ over $M$ such that $\eta|_{M^{(7)}}$ is stably isomorphic to a stable complex structure of $TM|_{ M^{ (7) } }$ and
$$w_{8}(\eta) - w_{8}(M) \in \Sq^{2} ( \rho_{2}(H^{6}(M)) ).$$
\end{proposition}

\begin{proposition}\label{prop:w8eta}
Let $M$ be as in Theorem \ref{thm:betad} and $\eta$ be a complex vector bundle over $M$ such that $\eta|_{M^{(7)}}$ is stably isomorphic to a stable complex structure of  $TM|_{ M^{(7)} } $. Then
$$ [w_{8}(\eta) - \rho_{2}(c_{1}(\eta) c_{3}(\eta)/2) ] = [0] \in H^{8}(M; \Z/2) / \Sq^{2} ( \rho_{2} (H^{6}(M)) ).$$
\end{proposition}

\begin{remark} \label{rem:c1c3}
Since $\eta|_{M^{(7)}}$ is stably isomorphic to a stable complex structure of $TM|_{ M^{(7)} }$, 
it follows that $\rho_{2}(c_{1}(\eta)) = w_{2}(M)$ and $\rho_{2}(c_{3}(\eta)) = w_{6}(M)$.
Hence, by Proposition \ref{prop:sw}, it make sense to divide $c_{1}(\eta)c_{3}(\eta)$ by $2$ and 
our statement makes sense by Lemma \ref{lem:sq}(c).
\end{remark}

\begin{proof}[Proof of Theorem \ref{thm:betad}]
One direction can be deduced easily from Propositions \ref{prop:8obs} and \ref{prop:w8eta}.
In the other direction, suppose that 
\begin{equation}\label{eq:w8m}
w_{8}(M) - \rho_{2}(cv/2) \in \Sq^{2} ( \rho_{2}(H^{6}(M)) )
\end{equation}
holds for some integral lifts $c \in H^{2}(M)$ and $v \in H^{6}(M)$ of $w_2(M)$ and $w_6(M)$ respectively.
Since $M$ is spin$^{c}$-able, 
$M$ admits a stable almost complex structure $\eta^{\prime}$ over $M^{(7)}$ by Lemma \ref{lem:M7}.
Since $\pi_{7}(BU) = 0$, it follows that $\eta^{\prime}$ can be extended to $\mdot$. 
Moreover, it follows from Lemma \ref{lem:8ext9} that $\eta^{\prime}$ can be extended to $M$, which will be denoted by $\eta$. 
That is, $\eta$ is a complex vector bundle over $M$ such that $\eta|_{M^{(7)}}$ is a stable complex structure of $TM|_{ M^{(7)} }$.
Therefore,
\begin{equation}\label{eq:w8eta}
w_{8}(\eta) - \rho_{2}(c_{1}(\eta) c_{3}(\eta)/2) \in \Sq^{2} ( \rho_{2}(H^{6}(M)) )
\end{equation}
by Proposition \ref{prop:w8eta}.
Note that $\rho_{2}(c_{1}(\eta) c_{3}(\eta)/2) - \rho_{2}(cv/2) \in \Sq^{2} ( \rho_{2}(H^{6}(M)) )$  
Lemma \ref{lem:sq}(a)(c)(e).
Hence,
\[ w_{8}(\eta) -  w_{8}(M) \in \Sq^{2} ( \rho_{2}(H^{6}(M)) ) \]
by \eqref{eq:w8m} and \eqref{eq:w8eta}.
Thus, by Proposition \ref{prop:8obs}, $M$ admits a stable complex structure over $\mdot$.
\end{proof}

\begin{proof}[Proof of Proposition \ref{prop:8obs}]
One direction of the proof is trivial by Lemma \ref{lem:8ext9}.
Now we assume that $\eta$ is a complex vector bundle over $M$ such that $\eta|_{M^{(7)}}$ is stably isomorphic to a stable complex structure of $TM|_{ M^{ (7) } }$ and
\[
w_{8}(\eta) - w_{8}(M) \in \Sq^{2} ( \rho_{2}(H^{6}(M)) ).
\]
Let $\zeta = \eta_{\R} {-} TM$. 
Then it is suffices to show that there exists a complex vector bundle over $M$ such that the underlying real vector bundle is isomorphic to $\zeta$ in the stable range. 
Since $\eta|_{M^{(7)}}$ is stably isomorphic to a stable complex structure of $TM|_{ M^{(7)} }$, it follows that $\zeta$ is stably trivial over $M^{(7)}$, hence is a spin vector bundle with 
$q_{1}(\zeta) = 0.$
Then it follows from  Li and Duan \cite[Theorem 1]{ld91} that there must exists an element $v \in H^{8}(M)$ such that 
\[
q_{2}(\zeta) = 3v.
\]
Therefore,   
\[
\rho_{2} (v) = \rho_{2} (q_{2}(\zeta)) = w_{8}(\zeta) = w_{8}(\eta) {-} w_{8}(M) \in \Sq^{2} ( \rho_{2} (H^{6}(M)) )
\]
by Equation \eqref{eq:q2} and our assumption.
Thus, according to Adams \cite{ad61}, there must exists a complex vector bundle $\gamma$ over $M$, trivial over $M^{(5)}$, such that 
\begin{align*}
c_{3}(\gamma) &= 2u,\\
c_{4}(\gamma) &= 3v = q_{2}(\zeta),
\end{align*}
for some $u \in H^{6}(M)$ satisfying $\Sq^{2} ( \rho_{2} (u) ) = \rho_{2} (v)$. 
Since $\pi_{i}(BO) = 0$ for $i=6,7$, it follows that $\gamma_{\R}$ must be stably trivial over $M^{(7)}$ and moreover
\[
q_{2}(\gamma_{\R}) = c_{4}(\gamma) = q_{2}(\zeta)
\]
by Equation \eqref{eq:cq}.
Note that both $\zeta$ and $\gamma_{\R}$ are stably trivial over $M^{(7)}$. Hence 
$\zeta - \gamma_{\R}$ is stably trivial over $M^{(7)}$ with 
\[
q_{2}(\zeta - \gamma_{\R}) = q_2(\zeta) - q_{2}(\gamma_{\R})  = 0
\]
by Thomas \cite[(1.10)]{th62}. 
Thus, it follows by a result of the second author \cite[Proposition 2.2]{ya18} that $\zeta - \gamma_{\R}$ admits a stable complex structure over $\mdot$; hence the same holds for $\zeta$ and $TM$. 
\end{proof}

\begin{proof}[Proof of Proposition \ref{prop:w8eta}]
According to Lemma \ref{lem:dm}, it suffices to prove that
\begin{equation*}
x \cdot (w_{8}(\eta) {-} \rho_{2}(c_1(\eta) c_2(\eta)/2)) = 0,
\end{equation*}
for any $x \in \mathcal{D}_{M}$. 
Let 
$\wh A(M)$ be the $A$-class of $M$ (cf.\ Atiyah and Hirzebruch \cite{ah59}); 
more precisely we have
\begin{equation}\label{eq:Aclass}
\wh A(M) = 1 - \frac{p_{1}(M)}{24} + \frac{-4p_{2}(M) + 7p_{1}^{2}(M)}{5760}.
\end{equation}
Denote by $l_{\eta}$ the complex line bundle over $M$ with 
$c_{1}(l_{\eta}) = c_{1}(\eta)$.
We will set 
\begin{equation*}
\wh A(M, \eta) = \wh A(M) \cdot e^{-c_{1}(\eta)/2} \cdot [ch(\eta) - ch(l_{\eta}) - \dim\eta + 1].
\end{equation*}
Here $ch$ denotes the Chern character.
Since $\eta|_{M^{(7)}}$ is stably isomorphic to a stable complex structure of $TM|_{ M^{(7)} }$, it follows that 
\begin{align}
w_{2}(M) & =  \rho_{2}(c_{1}(\eta)),  \label{eq:w2}\\
p_{1}(M)   & = c_{1}^{2}(\eta) - 2c_{2}(\eta). \label{eq:p1}
\end{align}
Therefore, it follows from Equation \eqref{eq:Aclass} and
\begin{align*}
e^{-c_{1}(\eta)/2}  = & ~ 1 - \frac{c_{1}(\eta)}{ 2} + \frac{c_{1}^{2}(\eta)}{8} -  \frac{c_{1}^{3}(\eta)}{48} + \frac{c_{1}^{4}(\eta)}{384}, \\
ch(\eta)  = &  ~ \dim\eta + c_{1}(\eta) + \frac{c_{1}^{2}(\eta) - 2c_{2}(\eta)}{2} + \frac{c_{1}^{3}(\eta) - 3c_{1}(\eta)c_{2}(\eta) + 3c_{3}(\eta)}{6} \\
 & ~ + \frac{c_{1}^{4}(\eta) - 4c_{1}^{2}(\eta)c_{2}(\eta) + 2c_{2}^{2}(\eta) + 4c_{1}(\eta)c_{3}(\eta) - 4c_{4}(\eta)}{24}, \\
ch(l_{\eta})  = & ~ 1 + c_{1}(\eta) + \frac{c_{1}^{2}(\eta)}{2} + \frac{c_{1}^{3}(\eta)}{6} + \frac{c_{1}^{4}(\eta)}{24},
\end{align*}
that for any $\bar x \in H^{1}(M)$, 
\begin{align*}
  & ~ \langle \bar x \cdot \wh A(M, \eta), [M] \rangle  \\
=   & ~ \langle \bar x \cdot [ \frac{c_{2}(\eta)p_{1}(M)}{24} - \frac{c_{1}^{2}(\eta)c_{2}(\eta)}{24} + \frac{c_{2}^{2}(\eta)}{12} - \frac{c_{1}(\eta) c_{3}(\eta)}{12} - \frac{c_{4}(\eta)}{6}], [M] \rangle.
\end{align*}
Hence,
\begin{equation} \label{eq:Agenus} 
 \langle \bar x \cdot \wh A(M, \eta), [M] \rangle 
=  - \langle \bar x \cdot [  \frac{c_{1}(\eta) c_{3}(\eta)}{12} + \frac{c_{4}(\eta)}{6}], [M] \rangle,
\end{equation}
by Equation \eqref{eq:p1}.

Now, consider the Cartesian product $M {\times} S^{1}$ of $M$ and the circle $S^{1}$.
Denote by 
$\pi_{1}\colon M \times S^{1} \rightarrow M$
and 
$\pi_{2}\colon M \times S^{1} \rightarrow S^{1}$
the projections to each factor respectively. 
Let $s \in H^{1}(S^{1})$ be a generator.
For any $x \in \mathcal{D}_{M}$,
it follows from $\boc (x) = 0$ that there exists $\bar x \in H^{1}(M)$ such that $\rho_{2}(\bar x) = x$. 
Denote by $l_{x}$ the complex line bundle over $M {\times} S^{1}$ with 
\[
c_{1}(l_{x}) = \pi_{2}^{\ast} (s) \cdot \pi_{1}^{\ast} (\bar x).
\]
Obviously,
\begin{equation} \label{eq:chlx}
ch(l_{x}) = 1 + c_{1}(l_{x}) = 1 + \pi_{2}^{\ast} (s) \cdot \pi_{1}^{\ast} (\bar x).
\end{equation}
Since $S^{1}$ is parallelizable, 
it follows that
\begin{align} 
\wh A(M \times S^{1})  & = \pi_{1}^{\ast} ( \wh A(M) ), \label{eq:ams} \\
w_{2}(M\times S^{1}) & = \pi_{1}^{\ast} (w_{2}(M)) = \rho_{2} (\pi_{1}^{\ast} (c_{1}(\eta))).  \label{eq:w2s}
\end{align} 
Therefore, it follows from Equations \eqref{eq:chlx} and \eqref{eq:ams} that
\begin{align}
C = & ~\langle \wh A(M\times S^{1}) \cdot e^{-\pi_{1}^{\ast}c_{1}(\eta)/2} \cdot [ch(l_{x}) {-} 1] \cdot [ch(\pi_{1}^{\ast}\eta)  {-} ch(\pi_{1}^{\ast} (l_{\eta})) {-} \dim\eta + 1], [M\times S^{1}]\rangle \notag \\
= &~ \langle \pi_{2}^{\ast} (s) \cdot \pi_{1}^{\ast} ( \bar x \cdot \wh A(M, \eta) ), [M \times S^{1}] \rangle \notag \\
= & ~ \langle \bar x \cdot \wh A(M, \eta), [M] \rangle.  \label{eq:A}
\end{align}
Recall that $\eta$, $l_{\eta}$ and $l_{x}$ are complex vector bundles. Then Equation \eqref{eq:w2s} and the differential Riemann-Roch theorem (cf.\ Atiyah and Hirzebruch \cite[Corollary 1]{ah59}) implies that the number $C$ is an integer.
Hence, 
\[
\langle \bar x \cdot (\frac{1}{2}c_{1}(\eta)c_{3}(\eta) + c_{4}(\eta)), [M] \rangle = - ~ 6C
\]
is an even integer by Equations \eqref{eq:Agenus} and \eqref{eq:A}. 
Thus, note that it make sense to divide $c_{1}(\eta)c_{3}(\eta)$ by $2$,
\[
x \cdot [\rho_{2}(\frac{c_{1}(\eta) c_{3}(\eta)}{2}) + w_{8}(\eta)] = \rho_{2} \left( \bar x \cdot (\frac{1}{2}c_{1}(\eta)c_{3}(\eta) + c_{4}(\eta)) \right) = 0.
\]
Then, since this identity holds for any $x \in \mathcal{D}_{M}$, it follows that
\[ w_{8}(\eta) - \rho_{2}(c_{1}(\eta) c_{3}(\eta)/2) \in \Sq^{2} ( \rho_{2} (H^{6}(M)) )\]
by Lemma \ref{lem:dm}.
\end{proof}

\subsection{The spinnable case}
In this subsection we assume that $M$ is a spinnable $9$-manifold. 
%
%
Since $w_{2}(M) = 0$, we have
$\Sq^{2} ( \rho_{2} (H^{6}(M)) )= 0$ by Lemma \ref{lem:sq}$(a)$. 
Hence from Lemma \ref{lem:obs} we immediately obtain

\begin{lemma}\label{lem:sq2}
The obstruction $\mathfrak{o}_{8}(h) \in H^{8}(M; \Z/2)$ does not depend on the choice of $h$. \qed
\end{lemma}

%
%

\noindent
Lemma \ref{lem:sq2} shows that $\mathfrak{o}_8(M)$ is a singleton 
and the main result of the subsection is the following

\begin{theorem}\label{thm:spin}
We have $\mathfrak{o}_{8}(M) = w_{8}(M)$.  Hence
$M$ admits a stable complex structure over $\mdot$ if and only if 
$w_{8}(M) = 0.$
\end{theorem}


\noindent
We outline the strategy for the proof of Theorem \ref{thm:spin} 
and then proceed to proofs of the required lemmas.

\begin{proof}[Proof of Theorem \ref{thm:spin} modulo Lemmas \ref{lem:spin}, \ref{lem:o8spin} and \ref{lem:sur}]
Recall that $\R P^{\infty}$ denotes infinite real projective space.
We consider a class $x \in H^8(M; \Z/2)$, which we
regard as a map $x \colon M \to \R P^\infty$
and we identify $x = x^*(u)$ where $u \in H^1(\R P^\infty; \Z/2)$ is the generator.
We must show that 
\[ x \cdot w_8(M) = x \cdot \mathfrak{o}_8(M) \]
for all $x \in H^8(M; \Z/2)$.
By Lemma \ref{lem:spin}, this holds if $x = \rho_2(\bar x)$ is the reduction of an integral
class $\bar x \in H^1(M)$.
To reduce the case of general $x$ to the integral case, we chose a spin structure
on $M$ and use computations in 9-dimensional spin bordism.
 
The pair $(M, x)$ defines an element of the singular spin bordism group 
$\Omega^{Spin}_9(\R P^\infty)$ and we consider the homomorphism
\[
\cdot w_{8} \colon \Omega_{9}^{Spin}(\R P^{\infty}) \to \Z/2,
\quad [M, x] \mapsto \an{x \cdot w_8(M), [M]}
\] 
We also consider the assignment 
\[
\cdot \mathfrak{o}_{8} \colon \Omega_{9}^{Spin}(\R P^{\infty}) \rightarrow \Z/2,
\quad [M, x] \mapsto \an{x \cdot \mathfrak{o}_{8}(M), [M]}.
\]
and in Lemma \ref{lem:o8spin} we prove that $\cdot \mathfrak{o}_{8}$ is a well defined homomorphism.
If $i \colon S^1 \to \R P^\infty$ denotes the inclusion of $S^1 = \R P^1$ then we have
the induced homomorphism $i_{\ast} \colon \Omega_{9}^{Spin}(S^{1}) \to \Omega_{9}^{Spin}(\R P^{\infty})$ 
as well as the compositions
$
\cdot \mathfrak{o}_{8} \circ i_{\ast} \colon \Omega_{9}^{Spin}(S^{1}) \to \Omega_{9}^{Spin}(\R P^{\infty}) \to \Z/2
$
and
$
\cdot w_{8} \circ i_{\ast} \colon \Omega_{9}^{Spin}(S^{1}) \to \Omega_{9}^{Spin}(\R P^{\infty}) \to \Z/2.
$
By \cite[Proposition 4]{k99} every element in $\Omega_{9}^{Spin}(S^{1})$ can be represented by 
$(N, x)$ with $\pi_{1}(N) \cong \Z$.  Now by Lemma \ref{lem:spin} we have
\[
\cdot \mathfrak{o}_{8} \circ i_{\ast} = \cdot w_{8} \circ i_{\ast}.
\]
In Lemma \ref{lem:sur} we prove that $i_*$ is onto
and so
$\mathfrak{o}_{8}(M) = w_{8}(M)$ for any spin manifold $M$.
\end{proof}

\begin{lemma}\label{lem:spin}
If $x = \rho_2(\bar x)$ for $\bar x \in H^1(M)$, then $x \cdot w_8 = x \cdot \mathfrak{o}_8$.
\end{lemma}

\begin{proof}
Represent $\bar x$ as a map $\bar x \colon M \to S^1$, which we may assume is transverse to a point 
$\ast \in S^1$.
Then $N : = \bar x^{-1}(\ast) \subset M$ is a closed smooth $8$-manifold with trivial $1$-dimensional normal bundle
and so inherits a spin structure from $M$, once we trivialise the normal bundle.
Since the obstruction $\mathfrak{o}_8$ is natural for codimension-$0$-inclusions of submanifolds, we have
$$ x \cdot w_8(M) = \bar x \cdot w_8(M) = \an{w_8(N), [N]} = \an{\mathfrak{o}_8(N), [N]} = 
\bar x \cdot \mathfrak{o}_8(M) = x \cdot \mathfrak{o}_8(M) ,$$
where the third equality holds since $\mathfrak{o}_8 = w_8$ for the spin manifold $N$ by Heaps \cite[Corollary 1.3]{he70}.
\end{proof}

\begin{lemma}\label{lem:o8spin}
Let $W$ be a spinnable $10$-manifold with boundary $\partial W = M$ and $i_{M} \colon M \rightarrow W$ be the inclusion.
For any map $F \colon W \rightarrow \R P^{\infty}$ let $f := F \circ i_{M}$. 
For $u \in H^{1}(\R P^{\infty}; \Z/2)$ the generator, we have
\[
\langle f^{\ast}(u) \cdot \mathfrak{o}_{8}(M), [M] \rangle = 0.
\]
\end{lemma}

\begin{proof}
If $W_{7}(W) = 0$, then $W$ admits a stable complex structure $h$ over its $7$-skeleton $W^{(7)}$ and 
the obstruction $\mathfrak{o}_{8}(h) \in H^{8}(W; \Z/2)$ is defined.  Since 
$i^{\ast}(\mathfrak{o}_{8}(h)) = \mathfrak{o}_{8}(M)$ we have
\begin{align*}
\langle  f^{\ast}(u) \cdot \mathfrak{o}_{8}(M), [M] \rangle 
= \langle i_{M}^{\ast}F^{\ast} (u) \cdot i_{M}^{\ast}(\mathfrak{o}_{8}(h)), [M] \rangle 
= \langle F^{\ast}(u) \cdot \mathfrak{o}_{8}(h), i_{M\ast}([M]) \rangle 
= 0,
\end{align*}
because $i_{M\ast}([M]) = 0$.

Now assume that $W_{7}(W) \neq 0$. Since $W_{7}(M) = i_{M}^{\ast}(W_{7}(W)) = 0$ by Theorem \ref{thm:W7}, $W_{7}(W)$ lifts to $\ol{W}_{7}(W) \in H^{7}(W, M)$. Denote by 
\[
PD \colon H^{7}(W, M) \rightarrow H_{3}(W)
\]
the Poincar\'e-Lefschetz 
isomorphism, and 
\[
T \colon \Omega_{3}^{SO}(W) \rightarrow H_{3}(W)
\] 
the Thom-homomorphism given by 
$T([X, g]) = g_{\ast}([X])$
for any $[X, g] \in \Omega_{3}^{SO}(W)$. 
Since $\Omega_1^{SO} = 0$, it follows from the Atiyah-Hirzebruch Spectral Sequence
for $\Omega^{SO}_*(W)$ that $T$ is surjective.
Hence, there is a closed orientable $3$-dimensional manifold $X$ with a map $g \colon X \rightarrow W$ such that 
\[
g_{\ast}([X]) = PD(\ol{W}_{7}(W)).
\]
By general position, $g$ may be chosen to be an embedding in the interior of $W$: 
$g(X) \subset W {-} M$.
Let $N$ be a tubular neiberghood of $g(X)$ and set
\[
W^{\prime} := W - \mathrm{int}(N) 
\] 
to be the space obtained from $W$ by removing the interior of $N$. Then $\partial W^{\prime} = M \coprod \partial N$. We denote by $i_{\partial N} \colon \partial N \to W$  the inclusion and set $f_{\partial N} := F \circ i_{\partial N}$ the composition map.

By construction, both $W^{\prime}$ and $N$ are spinnable and satisfy 
$W_{7}(W^{\prime}) = 0$ and $W_{7}(N) = 0$. Therefore,  it follows from the same argument in the case $W_{7}(W) = 0$ that
\[
\langle f^{\ast} (u) \cdot \mathfrak{o}_{8}(M), [M] \rangle + \langle f_{\partial N}^{\ast} (u) \cdot \mathfrak{o}_{8}(\partial N), [\partial N] \rangle  = 0,
\]
and 
\[
\langle f_{\partial N}^{\ast} (u) \cdot \mathfrak{o}_{8}(\partial N), [\partial N] \rangle = 0.
\]
Hence $\langle f^{\ast} (u) \cdot \mathfrak{o}_{8}(M), [M] \rangle = 0$.
\end{proof}

\begin{lemma}\label{lem:sur}
The induced homomorphism 
$
i_{\ast} \colon \Omega_{9}^{Spin}(S^{1}) \rightarrow \Omega_{9}^{Spin}(\R P^{\infty})
$
is surjective.
\end{lemma}


\begin{proof}
Consider the Atiyah-Hirzebtuch spectral sequence 
\[
\bigoplus_{p + q = n} H_{p}(\R P^{\infty};\Omega_{q}^{Spin}) \implies 
\Omega_{n}^{Spin}(\R P^{\infty}).
\]
In order to prove 
that $i_{\ast}$ is surjective we must show that $E_{p,9-p}^{\infty} := 
H_p(\R P^\infty; \Omega^{Spin}_q) = 0$ for $p \neq 0, 1$.
Recall that in low dimensions the groups $\Omega_{q}^{Spin}$ are given by the following table
(cf. Stong \cite[Item 2]{st86}):
\begin{table}[ht]
\centering
\begin{tabular}{c|*{9}{p{0.8cm}<{\centering}|}p{0.8cm}<{\centering}}
$q$ & 0 & 1  & 2 & 3 & 4 & 5 & 6 & 7 & 8 & 9 \\
\hline
$\Omega_{q}^{Spin}$ & $\Z$ & $\Z/2 $ & $\Z/2$  & $0$ & $\Z$ & $0$ & $0$ & $0$ & $\Z^{2}$ 
& $(\Z/2)^{2}$
\end{tabular}
\end{table}\\
Therefore, we only need to prove that $E_{p,9-p}^{\infty} = 0$ for $p = 5,~7,~8,~9$.

On the $E^{2}$ page, since the differential 
\begin{align*}
d_{2}\colon H_{p}(\R P^{\infty}; \Omega_{1}^{Spin}) \rightarrow H_{p-2}(\R P^{\infty}; \Omega_{2}^{Spin})
\end{align*}
is the dual of the Steenrod square 
\[
\Sq^{2}\colon H^{p-2}(\R P^{\infty}; \Z/2) \rightarrow H^{p}(\R P^{\infty}; \Z/2)
\]
and the differential 
\begin{align*}
d_{2}\colon H_{p}(\R P^{\infty}; \Omega_{0}^{Spin}) \rightarrow H_{p-2}(\R P^{\infty}; \Omega_{1}^{Spin})
\end{align*}
is $\bmod ~2$ reduction composed with the dual of $\Sq^{2}$
(cf. Teichner \cite[Lemma 2.3.2]{te92}), it follows that the differentials $d_{2}\colon E_{8, 1}^{2} \rightarrow E_{6, 2}^{2}$, $d_{2}\colon E_{9, 1}^{2} \rightarrow E_{7, 2}^{2}$ and $d_{2}\colon E_{9, 0}^{2} \rightarrow E_{7, 1}^{2}$ are all isomorphisms and
that $d_{2}\colon E_{10, 1}^{2} \rightarrow E_{8, 2}^{2}$ is trivial. Hence
\[
E_{7,2}^{\infty} = E_{7,2}^{3} = 0,~
E_{8,1}^{\infty} = E_{8,1}^{3} =0, ~
E_{9,0}^{\infty} = E_{9,0}^{3} =0, 
\]
and $E_{8,2}^{3} = E_{8, 2}^2 = \Z/2$.
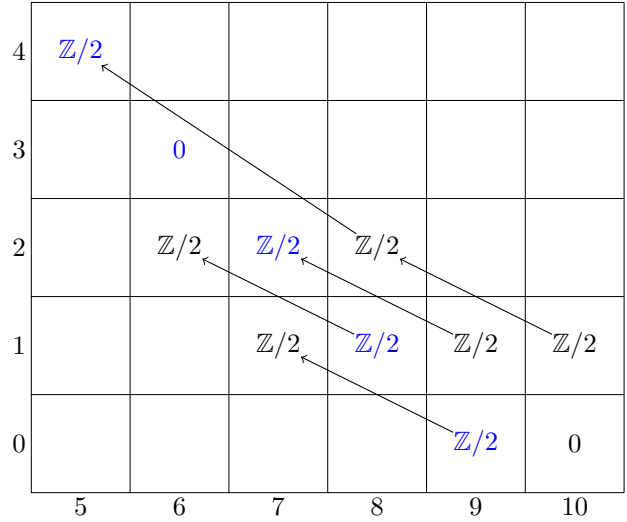
\sseqentrysize=1.3cm
\sseqxstep=1
\sseqystep=1
\begin{center}
\begin{sseq}{5...10}{0...4}
\ssmoveto{5}{4} \ssdrop[color=blue]{\Z/2}  
\ssmoveto{6}{3} \ssdrop[color=blue]{0}
\ssmoveto{6}{2} \ssdrop{\Z/2} 
\ssmoveto{7}{2} \ssdrop[color=blue]{\Z/2} 
\ssmoveto{8}{2} \ssdrop{\Z/2} \ssarrow{-3}{2}
\ssmoveto{7}{1} \ssdrop{\Z/2}
\ssmoveto{8}{1} \ssdrop[color=blue]{\Z/2} \ssarrow{-2}{1}
\ssmoveto{9}{1} \ssdrop{\Z/2} \ssarrow{-2}{1}
\ssmoveto{10}{1} \ssdrop{\Z/2} \ssarrow{-2}{1}
\ssmoveto{9}{0} \ssdrop[color=blue]{\Z/2} \ssarrow{-2}{1}
\ssmoveto{10}{0} \ssdrop{0}
\end{sseq}
\end{center}

On the $E^{3}$ page, the differential 
$d_{3} \colon E_{8,2}^{3} \rightarrow E_{5,4}^{3}$ is given by $\beta_{\Z/2} \circ Sq^2_*$,
by Proposition \ref{prop:AHSS-d3} below.
If $u \in H^1(\R P^\infty; \Z/2)$ is the generator, then by the Cartan formula
$$ Sq^2(u^6) = \left(Sq^1(u^3) \right)^2 = u^8.$$
Hence $Sq^2_* \colon H_8(\R P^\infty; \Z/2) \to H_6(\R P^\infty; \Z/2)$ is an isomorphism.
Since the Bockstein homomorphism
$\beta_{\Z/2} \colon H_6(\R P^\infty; \Z/2) \to H_5(\R P^\infty)$ is also an isomorphism,
it follows that $d_{3} \colon E_{8,2}^{3} \rightarrow E_{5,4}^{3}$ is an isomorphism
and so
$E_{5,4}^{\infty} = E_{5,4}^{4} = 0$.
%

\end{proof}

We conclude this section with a general result on the $d_3$-differential from the $2$-line in the AHSS
for spin bordism.  This result was used in the proof of Lemma \ref{lem:sur} and is of
independent interest.

\begin{proposition} \label{prop:AHSS-d3}
Let the spin bordism AHSS for $X$,
$E^3_{i+3, 2} = H_{i+3}(X; \Z/2)/d_2(H_{i+5}(X; \Z/2))$ 
and the differential
$d_3 \colon E^3_{i+3, 2} \to E^3_{i, 4} = 
H_i(X)$ fits into the following commutative diagram
%
%
$$
\xymatrix{
H_{i+3}(X; \Z/2) \ar[d] \ar[r]^(0.475){Sq^2_*} &
H_{i+1}(X; \Z/2) \ar[r]^(0.575){\beta_{\Z/2}} &
H_i(X) \ar[d]^= \\
E^3_{i+3, 2} \ar[rr]^{d_3} & &
H_i(X),}$$
where we have identified $\Omega^{Spin}_1 = \Z/2 = \Omega^{Spin}_2$ and 
$\Omega_4^{Spin} \cong \Z$ 
and $\beta_{\Z/2}$ denotes the Bockstein associated to the short exact coefficient sequence
$\Z \xra{\times 2} \Z \to \Z/2$.
\end{proposition}

\begin{proof}
We first note that $\beta_{\Z/2} \circ Sq^2_* \circ Sq^2_* \colon H_{i+5}(X; \Z/2) \to H_i(X)$
vanishes.  This is because the Steenrod relation $Sq^2Sq^2 = Sq^3Sq^1$
implies that $Sq^2_* \circ Sq^2_* = Sq^1_* \circ Sq^3_*$ and because 
general $\beta_{\Z/2} \circ Sq^1_* = 0$ holds in general.
It follows that $\beta_{\Z/2} \circ Sq^2_* \colon H_{i+3}(X; \Z/2) \to H_i(X)$ 
induces a well-defined homomorphism $E^3_{i+3, 2} \to E^3_{i, 4}$.

To show that $d_3$ fits into the commutative diagram of the lemma
we let $\MSpin$ be the Thom spectrum for spin bordism
and $\MSpin\an{2} \to \MSpin$ its $1$-connected cover.
Now $\pi_i(\MSpin) = \Omega^{Spin}_i$ and so
$\pi_3(\MSpin) = \pi_5(\MSpin) = 0$ by \cite{mi65}.
It follows that the $d_3$-differentials from the
$2$-lines of the AHSSs for $\MSpin$ and $\MSpin\an{2}$ fit into the
following commutative diagram
$$
\xymatrix{
H_{i+3}(X; \Z/2) \ar[d] \ar[rr]^(0.6){d_3^{\MSpin\an{2}}} & &
H_i(X) \ar[d]^= \\
E^3_{i+3, 2} \ar[rr]^{d_3} & &
H_i(X),}$$
where the left vertical map $H_{i+1}(X; \Z/2) \to E^3_{i+1, 2}$ is surjective.
Hence we must show $d_3^{\MSpin\an{2}} = \beta_{\Z/2} \circ Sq^2_*$.
To do this, we consider the natural map $\MSpin\an{2} \to P_4(\MSpin\an{2})$ to 
the $4$th Postnikov stage of $\MSpin\an{2}$ as in \cite{ten17}
and we see that if suffices to consider the
$d_3$-differential from the $2$-line in the AHSS
for $\P_4\an{2} := P_4(\MSpin\an{2})$.
Now there is a fibration sequence of specta
\begin{equation} \label{eq:P4}
 \K(\Z, 4) \to \P_4\an{2} \to \K(\Z/2, 2),
\end{equation}
where $\K(A, q)$ denotes the Eilenberg-MacLane spectrum with $\pi_q(\K(A, q)) = A$
and it is elementary that the the $d_3$-differential 
$$ d_3^{\P_4\an{2}} \colon H_{i+3}(X; \pi_2(\P_4\an{2})) \to H_i\bigl( X; \pi_4(\P_4\an{2}) \bigr)$$
is given by the $k$-invariant, $\kappa \in H^5(\K(\Z/2, 2); \Z)$, of the fibration \eqref{eq:P4}.
Moreover, there is an isomorphism $H^5(\K(\Z/2, 2); \Z) \cong \Z/2(\boc(Sq^2(v)))$,
where $v \in H^2(\K(\Z/2, 2); \Z/2) = \Z/2$ is the generator and so we must show
that $\kappa$ is non-trivial. 

Given a spectrum $\M$ let $\Omega^\infty \M$ denote the corresponding infinite loop space.
To show $\kappa \neq 0$ we relate $\Omega^\infty\MSpin$ to the classifying spaces for surgery.
Recall that $BSO \to BSG$ is the map classifying the transition from oriented stable vector bundles 
to oriented stable spherical fibrations and $G/O$ be its homotopy fibre.
Similarly, let $G/PL$ be the homotopy fibre of $BSPL \to BSG$,
the analogous map for oriented stable piecewise linear bundles.
The space $G/O$ is an infinite loop space and there is a map
$c \colon G/O \to \Omega^\infty\MSpin$ such that
the induced map on homotopy groups 
$c_* \colon \pi_i(G/O) \to \pi_i(\Omega^\infty\MSpin) = \Omega^{Spin}_*$
is identified with the canonical map from almost framed bordism to spin bordism,
for $i \geq 2$, \cite[\S 2.3]{css18}.
It is well-known that this map is an isomorphism
on homotopy groups $i = 2, \dots, 7$.  Since $G/O$ is $1$-connected, it follows that
the induced map $P_4(G/O) \to \Omega^\infty\P_4\an{2}$ is a homotopy equivalence.
Now the fibre of the map $G/O \to G/PL$ is $PL/O$, the fibre of the canonical map 
$BSPL \to BSO$, which 6-connected by \cite{ce68} and so the map $G/O \to G/PL$ is
$6$-connected.
Moreover, the first $k$-invariant of $G/PL$ is known \cite[Theorem 4.8]{mm79} to be non-trivial.  
It follows that the first $k$-invariant of $G/O$ is
non-trivial and so the same holds for $\Omega^\infty \P_4\an{2}$ and thus $\P_4\an{2}$.

\end{proof}



\section{The top obstruction}  \label{s:top}
In this section $M$ will be an $n$-manifold with $n \equiv 1 \bmod 8$. 
Recall that $\mdot := M - \mathrm{int}(D^n)$
is the space obtained from $M$ by removing the interior of a small $n$-disc in $M$.
Suppose that $\xi$ is a stable real 
vector bundle over $M$ whose restriction to $\mdot$ 
admits a stable complex structure $h$, so that the top obstruction, 
$$\mathfrak{o}_{8k+1}(h) \subset H^{8k+1}(M; \pi_{8k}(SO/U) = H^{8k+1}(M; \Z/2)$$
is defined.  In the first part of this section we use $K$-theory to show that $\mathfrak{o}_{8k+1}(h) = 0$
if $w_2(M) \neq 0$ and that $\mathfrak{o}_{8k+1}(h)$
can be identified with the topological index $\pi^1(\xi, s)$ (defined in Section \ref{ss:K-theory})
if $M$ is spinnable and $s$ is a spin structure on $M$.
In the second part of this section, we identify $\pi^1(M, s)$ with the invariant $\sigma_{w_4}(M)$ 
of the introduction.


\subsection{The top obstruction via $K$-theory} \label{ss:K-theory}
Let $KO$ and $K$ denote respectively real and complex $K$-theory.
Let
$r \colon K \rightarrow KO$
denote the real reduction map and also the real reduction homomorphism
\[
r \colon \wt{K}(M) \to \wt{KO}(M).
\]
By definition, $M$ admits a stable complex structure if and only if the stable tangent bundle $\wt{TM} \in \wt{KO}(M)$ lies in the image of $r$.
More generally, for a real vector bundle $\xi$ over $M$, $\xi$ admits a stable complex structure over $M$ if 
$$\wt{\xi} := \xi - \dim \xi \in \wt{KO}(M)$$
 lies in the image of $r$.

Recall that if $M$ is spinnable, then $M$ is $KO$-orientable (cf.\ Atiyah, Bott and Shapiro \cite[Theorem (12.3)]{abs64}). For a spin structure $s$ on $M$, we denote by 
$\{ M, s\} \in KO_{n}(M)$
a $KO$-orientation of $(M, s)$.
Following \cite[\S 4]{abp66} we define
\[
\pi^{1}(\xi, s)
: = \langle \,\wt{\xi}\, , \{M, s\} \rangle_{KO},
\]
where $\langle \, \cdot \, , \, \cdot \, \rangle_{KO} $ is the Kronecker product in real $K$-theory.
%
%
The main result of this subsection is the following

\begin{theorem}\label{thm:topo}
Let $\xi$ be a real vector bundle over $M$. Suppose that $\xi |_{\mdot}$ admits a stable complex structure. Then $\xi$ admits a stable complex structure if and only if one of the following conditions holds:
\begin{enumerate}
\item[(a)] $w_2(M) = 0$ and there exists a spin structure $s$ on $M$ such that $\pi^{1}(\xi, s)=0$;
\item[(b)] $w_{2}(M) \neq 0$.
\end{enumerate}
\end{theorem}

\begin{remark}\label{rem:pi1s}
If $w_2(M) = 0$ and $\xi$ admits a stable complex structure over $\mdot$, it follows immediately from Theorem \ref{thm:topo} and Proposition \ref{prop:pi1} that the value of 
\[
\pi^1(\xi, s)
\in KO_{n}(\pt) \cong \Z/2
\]
depends neither on the spin structure $s$, nor on the choice of $KO$-orientation $\{M, s\}$.
\end{remark}

We are most interested in the case where $ \xi = TM$ and in this case
we set 
$$\pi^1(M, s) := \pi^1(TM, s).$$
We recall from the Lemma \ref{lem:omegac} that
the value of $\Omega(p_c(M))$ does not depend on the choice of $c$.
Now from Theorems \ref{thm:topo}, \ref{thm:Omega} and \ref{thm:spin} 
we immediately obtain

\begin{theorem}\label{thm:mainpi}
Let $M$ be a closed orientable $9$-manifold. 
Then $M$ admits a contact structure if and only if one of the following conditions is satisfied:
\begin{enumerate}
\item[(a)] 
$M$ is spinnable, 
$w_8(M) = 0$
and there exists a spin structure $s$ on $M$ such that 
$\pi^{1}(M, s)=0;$
\item[(b)] $M$ is spin$^{c}$-able with 
$w_{2}(M)\neq0$ and 
$\Omega(p_c(M))=0,$
where $c \in H^{2}(M; \Z)$ is a spin$^{c}$ characteristic class of $M$, i.e., $\rho_{2} (c) = w_{2}(M)$.
\hfill
\qed
\end{enumerate}

\end{theorem}




Before proving Theorem \ref{thm:topo} we give an index-theoretic interpretation of $\pi^1(M, s)$.
Suppose that $M$ is spinnable with dimension $n \equiv 1 \bmod 8$. 
Let $\slashed{\mathfrak{D}}$ be the canonical $Cl_n$-linear Atiyah-Singer operator,
and let $\slashed{\mathfrak{D}}_{TM}$ be the $Cl_n$-linear Atiyah-Singer operator with coefficients in $TM$.
Denote by 
$$\ind_n(\slashed{\mathfrak{D}}_{TM}) \in KO^{-n}(\pt)
\quad \text{and} \quad
\ind_n(\slashed{\mathfrak{D}}) \in KO^{-n}(\pt)$$
the analytic index of $\slashed{\mathfrak{D}}_{TM}$ and 
$\slashed{\mathfrak{D}}$ respectively (\cite[Chapter II, \S 7]{lm89}).
If we regard $\pi^1(M,s)$ as an element in $KO^{-n}(\pt)$ under the isomorphism $KO^{-n}(\pt) \cong KO_n(\pt)$, then it follows from Lawson and Michelsohn \cite[p. 276, Theorems 16.6 and Theorem 16.8]{lm89} that

\begin{lemma} \label{lem:index}
$\pi^1(M,s) = \ind_n(\slashed{\mathfrak{D}}_{TM}) - \ind_n(\slashed{\mathfrak{D}})$.
\hfill \qed
\end{lemma}

We now turn to the proof of Theorem \ref{thm:topo}, which we reduce to the following
sequence of propositions and lemmas.
Recall that the complexification map 
$c \colon KO \rightarrow K$
is a ring morphism and it makes $K$ in to a $KO$-module spectrum. Moreover, the real reduction map $r$ is a $KO$-module morphism.
Therefore, we have the Kronecker product:
\[
\langle  \, \cdot \, ,  \, \cdot \, \rangle_{K} \colon K^{p}(M) \times KO_{q}(M) \rightarrow K_{q-p}(\pt),
\]
and it follows immediately from the definition of Kronecker product (cf. Adams \cite[p.\,233]{ad74b}) and $KO$-module morphism (cf. Rudyak \cite[p.\,52]{ru98b}) that
\begin{equation}\label{eq:Krn}
r (\langle x, y \rangle_{K}) = \langle r(x), y \rangle_{KO}
\end{equation}
holds for any $x \in K^{p}(M)$ and $y \in KO_{q}(M)$.

Denote by $i\colon \mdot \rightarrow M$ the inclusion map and $p \colon M \rightarrow S^{n}$ the map by collapsing $\mdot$ to a point. Consider the following exact ladder
\begin{equation}\label{eq:KOK}
\begin{split}
\xymatrix{
\widetilde{K}(S^{n}) \ar[r]^-{p^{\ast }} \ar[d]_{r} & \widetilde{K}(M) \ar[r]^-{i^{\ast}} \ar[d]_{r} & \widetilde{K}(\mdot) \ar[d]_{r}  \\
\widetilde{KO}(S^{n}) \ar[r]^-{p^{\ast}}  & \widetilde{KO}(M) \ar[r]^-{i^{\ast}}  & \widetilde{KO}(\mdot),  
}
\end{split}
\end{equation}
where $p^{\ast}$ and $i^{\ast}$ are the homomorphisms induced by $p$ and $i$ respectively. 

\begin{proposition}\label{prop:extc}
The homomorphism
$i^{\ast} \colon \widetilde{K}(M) \rightarrow \widetilde{K}(\mdot)$
is an isomorphism.
\end{proposition}


\begin{proof}
On the one hand, since $\wt{K}(S^{n}) = 0$ when $n\equiv 1 \bmod{8}$, it follows from the exact ladder \eqref{eq:KOK} that $i^{\ast}$ is injective.
On the other hand, for any $\eta \in \wt{K}(\mdot)$, 
it can be proved in much the same way as Lemma \ref{lem:8ext9} that $\eta$ can be extended to $M$, hence $i^{\ast}$ is surjective.
\end{proof}

\begin{proposition}\label{prop:pi1}
Suppose that $M$ is spinnable, $s$ is a spin structure on $M$ and $\{ M, s\} \in KO_{n}(M)$
is a $KO$-orientation of $(M, s)$.  Then
\begin{enumerate}
\item[(a)] for any $\zeta \in \widetilde{K}(M)$,  $\langle r(\zeta), \{M, s\} \rangle_{KO} = 0$, 
\item[(b)] for $\zeta \in \widetilde{KO}(S^{n})$, $\langle p^{\ast}(\zeta), \{M, s\} \rangle_{KO} = 0$ if and only if $\zeta = 0$.
\end{enumerate}
\end{proposition}

\begin{proof}
Note that we have $n \equiv 1 \bmod 8$.

$(a)$ Since $\langle \zeta, \{M, s\} \rangle_{K} \in K_{n}(pt) = 0$, 
$$\langle r(\zeta), \{M, s\} \rangle_{KO} = r(\langle \zeta, \{M, s\} \rangle_{K}) = 0$$
by Equation \eqref{eq:Krn}.

$(b)$ Denote by $s_{0}$ the spin structure on $S^{n}$ and $p_{\ast} \colon KO_{n}(M) \rightarrow KO_{n}(S^n)$ the induced homomorphism. 
Since $p_{\ast}(\{M, s\})$ is a $KO$-orientation of $S^{n}$ (cf. Rudyak \cite[p. 321, 2.12. Lemma]{ru98b}), denote it by $\{S^{n}, s_{0}\}$, it follows that
$$\langle p^{\ast}(\zeta), \{M, s\} \rangle_{KO} = \langle \zeta, p_{\ast}(\{M, s\}) \rangle_{KO} = \langle \zeta, \{S^{n}, s_{0}\} \rangle_{KO}.$$
Therefore $(b)$ follows immediately from the Poincare Duality.
\end{proof}

\begin{proposition}\label{prop:nspin}
$w_{2}(M) \neq 0$ if and only if
\begin{equation*}
\mathrm{Im}~[p^{\ast} \colon \widetilde{KO}(S^{n}) \rightarrow \widetilde{KO}(M)] \subset \mathrm{Im}~[r \colon \widetilde{K}(M) \rightarrow \widetilde{KO}(M)].
\end{equation*}
\end{proposition}


We have divided the proof of Proposition \ref{prop:nspin} into a sequence of lemmas as follows.
\begin{lemma}
Suppose that
$\mathrm{Im}~[p^{\ast} \colon \widetilde{KO}(S^{n}) \rightarrow \widetilde{KO}(M)] \subset \mathrm{Im}~[r \colon \widetilde{K}(M) \rightarrow \widetilde{KO}(M)]$.
Then $w_{2}(M) \neq 0$.
\end{lemma}

\begin{proof}
Conversely, suppose that $M$ is spinnable, i.e., $w_{2}(M) = 0$. Take $\zeta \in \widetilde{KO}(S^{n}) \cong \Z/2$ be the generator. 
Then 
$$\mathrm{Im}~[p^{\ast} \colon \widetilde{KO}(S^{n}) \rightarrow \widetilde{KO}(M)] \subset \mathrm{Im}~[r \colon \widetilde{K}(M) \rightarrow \widetilde{KO}(M)]$$
implies that there exists $\eta \in \wt{K}(M)$ such that $r(\eta) = p^{\ast}(\zeta)$. 
Hence 
$$\langle p^{\ast}(\zeta), \{M, s\}  \rangle_{KO} = \langle r(\eta), \{M, s\}  \rangle_{KO} = 0$$
by Proposition \ref{prop:pi1}$(a)$. Thus it follows from Proposition \ref{prop:pi1}$(b)$ that $\zeta = 0$, a contradiction.
\end{proof}

\begin{lemma}
The homomorphism $i^{\ast}\colon KO^{-1}(M)\rightarrow KO^{-1}(\mdot)$
is injective if and only if
\begin{equation*}
\mathrm{Im}~[p^{\ast} \colon \widetilde{KO}(S^{n}) \rightarrow \widetilde{KO}(M)] \subset \mathrm{Im}~[r \colon \widetilde{K}(M) \rightarrow \widetilde{KO}(M)].
\end{equation*}

\end{lemma}

\begin{proof}
Combining the Bott exact sequence (cf. Yang \cite{ya15}) and the exact sequences of the $K$ and $KO$ groups of the pair $(M, \mdot)$, we can get the following commutative diagram:
\begin{equation}
\begin{split}
\xymatrix{
\widetilde{K}(S^{n}) \ar[r]^-{p^{\ast }} \ar[d]_{r} & \widetilde{K}(M) \ar[r]^-{i^{\ast}} \ar[d]_{r} & \widetilde{K}(\mdot)  \ar[d]_{r} \\
\widetilde{KO}(S^{n}) \ar[r]^-{p^{\ast}} \ar[d]_{\gamma} & \widetilde{KO}(M) \ar[r]^-{i^{\ast}} \ar[d]_{\gamma} & \widetilde{KO}(\mdot)  \ar[d]_{\gamma}\\
KO^{-1}(S^{n}) \ar[r]^-{p^{\ast}} & KO^{-1}(M) \ar[r]^-{i^{\ast}} & KO^{-1}(\mdot).
}
\end{split}
\end{equation}
Then 
\begin{equation*}
\mathrm{Im}~[p^{\ast} \colon \widetilde{KO}(S^{n}) \rightarrow \widetilde{KO}(M)] \subset \mathrm{Im}~[r \colon \widetilde{K}(M) \rightarrow \widetilde{KO}(M)]
\end{equation*}
if and only if $\gamma \circ p^{\ast} =0$.
Since $\gamma \circ p^{\ast} = p^{\ast} \circ \gamma$ and $\gamma \colon \widetilde{KO}(S^{n}) \rightarrow KO^{-1}(S^{n})$
is an isomorphism, it follows that $\gamma \circ p^{\ast} =0$ if and only if $p^{\ast} \colon KO^{-1}(S^{n}) \rightarrow KO^{-1}(M)$ is the trivial homomorphism, i.e., $i^{\ast}\colon KO^{-1}(M)\rightarrow KO^{-1}(\mdot)$ is injective.
\end{proof}

\begin{lemma}
If $w_{2}(M) \neq 0$, then
$i^{\ast}\colon KO^{-1}(M)\rightarrow KO^{-1}(\mdot)$ is injective.
\end{lemma}

\begin{proof}
Let us consider the Atiyah-Hirzebruch spectral sequence for $KO^{\ast}(M)$. 

Since $n \equiv 1 \bmod 8$, 
$KO^{-n} \cong KO^{-n-1} \cong \Z/2$.
Note that the differential 
\[
d_{2}^{n-2, -n} \colon H^{n-2}(M; KO^{-n}) \rightarrow H^{n}(M; KO^{-n-1})
\]
on the $E_{2}$ page is the Steenrod square $\Sq^{2}$ (e.g.\ see Fujii \cite[formula (1.3)]{fu67}).
Therefore, $w_{2}(M) \neq 0$ implies that
$d_{2}^{n-2, -n} \neq 0$,
and hence
\begin{equation}\label{eq:einfty}
E_{\infty}^{n, -n-1} \cong E_{3}^{n, -n-1} = 0.
\end{equation}
Recall that by construction (cf.\ Switzer \cite[pp.\,336-341]{sw02b} or \cite{ya15}), 
\[
E_{\infty}^{n, -n-1} \cong F^{n, -n-1}/F^{n+1, -n-2},
\]
where 
$F^{p,q}:= \mathrm{Ker}~[i^{\ast}\colon KO^{p+q}(M)\rightarrow KO^{p+q}(M^{p-1})]$.
Consequentely, since $F^{n+1, -n-2} = 0$ by definition, it follows  that
$F^{n, -n-1} = 0$. That is,
\[
i^{\ast}\colon KO^{-1}(M)\rightarrow KO^{-1}(\mdot)
\]
is injective.
\end{proof}

\begin{proof}[Proof of Theorem \ref{thm:topo}]
The theorem
follows by diagram chasing in the commutative diagram \eqref{eq:KOK} and
applying Propositions \ref{prop:extc}, \ref{prop:pi1} and \ref{prop:nspin}.
\end{proof}

\subsection{The top obstruction and $\sigma_{w_4}$} \label{ss:sigma_and_pi}
By Theorem \ref{thm:mainpi}, if $M$ is a spinnable $9$-manifold with $w_8(M) = 0$,
then the top obstruction to $M$ admitting a stable complex structure is given by
$\mathfrak{o}_9(M) = \pi^1(M, s)$.  We note that the assignment
$$ \pi^1 \colon \Omega^{Spin}_9 \to \Z/2, 
\quad [M, s] \mapsto \pi^1(M, s),$$
is a well-defined bordism invariant by \cite{abp66}.
In \cite{c20} it was shown that map
$\sigma_{w_4} \colon \Omega^{Spin}_9 \to \Z/2, [M, s] \mapsto 
\sigma_{w_4}(M)$, defines a non-trivial invariant of spin bordism
and that 
the kernel of $\sigma_{w_4}$ is precisely those bordism classes which 
contain a homotopy $9$-sphere.

\begin{proposition} \label{prop:pi_and_sigma}
We have $\pi^1 = \sigma_{w_4} \colon \Omega^{Spin}_9 \to \Z/2$.
\end{proposition}

\begin{proof}
Let $\Omega^{SU}_9$ be the $9$-dimensional $SU$-bordism group and $F_{\ast} \colon \Omega^{SU}_9 \to \Omega^{Spin}_9$ the forgetful homomorphism. We first show that there is a short exact sequence
\begin{equation} \label{eq:SU_to_Spin}
0 \to \Omega^{SU}_9 \xrightarrow{~F_{\ast}~}
 \Omega^{Spin}_9 \xrightarrow{~\pi^1~} \Z/2 \to 0.
\end{equation}
It is well known that $\Omega^{SU}_9 = \Z/2$, $\Omega^{Spin}_9 \cong ( \Z/2 )^2$
and the forgetful homomorphism $\Omega^{SU}_9 \to \Omega^{Spin}_9$ is monic 
(see Stong \cite[p.\,\,353]{st68b}). 
Secondly, we prove that $\mathrm{Ker}(\pi^1) = \mathrm{Im}(F_{\ast})$ . 
From Theorem \ref{thm:topo}, we easily deduce that $\pi^1 \circ F_{\ast} = 0$ and so
$\mathrm{Im}(F_{\ast}) \subseteq \mathrm{Ker}(\pi^1)$. 
Now for any spin cobordism class $[M, s] \in \Omega^{Spin}_9$ 
suppose that $\pi^1( [M, s] ) = 0$. Since $M$ can be selected to be a $3$-connected manifold by \cite[Proposition 2.6]{bcs14}, it follows from Theorems \ref{thm:spin} and \ref{thm:topo} 
that $M$ admits a stable almost complex structure, hence an $SU$-structure. 
This means that $[M, s] \in \mathrm{Im}(F_{\ast})$.  
Hence $\mathrm{Ker}(\pi^1) \subseteq \mathrm{Im}(F_{\ast})$,
$\mathrm{Ker}(\pi^1) = \mathrm{Im}(F_{\ast})$ and
we have shown that \eqref{eq:SU_to_Spin} is exact.

The $\alpha$-invariant $\alpha \colon \Theta_9 \to \Z/2$ is onto \cite[Theorem 2]{mi65}
and since $\Omega^{SU}_9 = \Z/2$ and every homotopy $9$-sphere admits an $SU$-structure,
it follows that  the natural map
$\Theta_9 \to \Omega^{SU}_9$ is onto.
Hence the exact sequence \eqref{eq:SU_to_Spin} fits into the following commutative diagram of exact sequences:
$$
\xymatrix{ 
&
\Theta_9 \ar[r] \ar[d] &
\Omega^{Spin}_9 \ar[d]^= \ar[r]^{\sigma_{w_4}} &
\Z/2 \ar[d]^= \ar[r] &
0 \\
0 \ar[r] &
\Omega^{SU}_9 \ar[r] &
\Omega^{Spin}_9 \ar[r]^{\pi^1}&
\Z/2 \ar[r] &
0
}$$
We see that $\pi^1, \sigma_{w_4} \colon \Omega^{Spin}_9 \to \Z/2$ 
are non-zero homomorphisms with $\ker(\sigma_{w_4}) \subseteq \ker(\pi^1)$.
Hence we have $\ker(\sigma_{w_4}) = \ker(\pi^1)$ and so
$\pi^1 = \sigma_{w_4}$.
\end{proof}

We conclude with an 
example of a spinnable $3$-connected $9$-manifold which does not admit a contact structure.
Let $(M_1, s_1)$ be the manifold from Section \ref{ss:examples}, which was obtained by
spin surgery along $S^1 \times D^8 \subset S^1 \times \Ha P^2$, where $S^1 \times \Ha P^2$
has a product spin structure with the spin structure on $S^1$ not bounding.
By \cite[Theorem p.\,339]{st68b}, $[M_1, s_1] \neq 0 \in \Omega^{Spin}_9$, but since the $\wh A$-genus
of $\Ha P^2$ vanishes, the $\alpha$-invariant of $(M_1, s_1)$ vanishes.
Now the proof of Proposition \ref{prop:pi_and_sigma} shows that the image of $\Omega^{SU}_9$
in $\Omega^{Spin}_9$ is detected by the $\alpha$-invariant.  It follows that 
\begin{equation} \label{eq:M_1}
\pi^1(M_1, s_1) = w_4(M) \cup \wh \phi(M) \neq 0
\end{equation}
and thus $M_1$ does not admit a contact structure.  However, it is clear from the surgery
construction of $M_1$ that it is $3$-connected; indeed $M_1$ is simply-connected
and $H_*(M_1) \cong H_*(S^5 \times S^4)$.



\bibliographystyle{amsplain}

\begin{multicols}{2}[]
\bigskip
\noindent
\emph{Diarmuid Crowley}\\
  \vskip -0.125in \noindent
  {\small
  \begin{tabular}{l}%
    School of Mathematics and Statistics\\
    University of Melbourne\\
    Parkville, VIC, 3010, Australia\\
    \textsf{dcrowley@unimelb.edu.au}
  \end{tabular}}

\noindent
\emph{Huijun Yang}\\  
\vskip -0.125in \noindent
{\small
  \begin{tabular}{l}%
   School of Mathematics and Statistics\\
   Henan University\\
   Kaifeng, Henan, 475004, China\\
    \textsf{yhj@amss.ac.cn}
  \end{tabular}}
\end{multicols}

\end{document}